\documentclass[a4paper]{article}

\usepackage[latin1]{inputenc}
\usepackage[T1]{fontenc}

\usepackage[margin=3cm]{geometry}         
\usepackage[english]{babel}

\usepackage{mathrsfs,amsthm,amssymb,amsfonts,amsmath,color,graphicx,float,url,bbm,multirow,caption,tikz,hyperref}
\usepackage[authoryear,round,sort]{natbib}

\usepackage{epstopdf}
\newtheorem{lem}{Lemma}[section]
\newtheorem{cor}{Corollary}[section]

\newtheorem{pro}{Proposition}[section]
\newtheorem{theo}{Theorem}[section]
\theoremstyle{remark}
\newtheorem{rem}{Remark}

\newcommand{\C}{\mathcal{C}}
\renewcommand{\L}{\mathscr{L}}

\renewcommand{\P}{\mathbb{P}}
\newcommand{\e}{\varepsilon}

\newcommand{\llangle}{\left\langle}
\newcommand{\rrangle}{\right\rangle}

\newcommand{\E}{\mathbb{E}}

\newcommand{\I}{\mathbf{1}}
\newcommand{\R}{\mathbb{R}}
\newcommand{\N}{\mathbb{N}}

\DeclareMathOperator*{\argmin}{\arg\min}

\newcommand{\Bigpar}[1]{\Bigl(#1\Bigr)}
\newcommand{\bigpar}[1]{\bigl(#1\bigr)}
\newcommand{\biggpar}[1]{\biggl(#1\biggr)}

\newcommand{\bigcro}[1]{\bigl[#1\bigr]}

\newcommand{\biggcro}[1]{\biggl[#1\biggr]}

\let\phi=\varphi

\definecolor{vert1}{rgb}{0.0, 0.5, 0.0}

\title{On principal curves with a length constraint}

\author{Sylvain Delattre \& Aurélie Fischer\footnote{The research work of this author has been partially supported by the French National Research Agency via the TopData project ANR-13-BS01-0008}}

\newcommand{\Addresses}{{
		\vspace{1.5cm}
		\footnotesize
				\noindent Sylvain Delattre \& Aurélie Fischer\\ Laboratoire de Probabilités et Modèles Aléatoires\\
			Université Paris Diderot\\Bâtiment Sophie Germain\\
			Case courrier 7012\\75205 Paris Cedex 13, France
		
}}

\begin{document}
	\maketitle

	\begin{abstract}
			In this paper, we are interested in the problem of finding a parametric curve $f$ minimizing the quantity  $\E\left[\min_{t\in[0,1]}\|X-f(t)\|^2\right]$, where $X$ is a random variable,  under a length constraint. This question is known in the probability and statistical learning  context as length-constrained principal curves optimization, as introduced by \cite{KKLZ}, and it also corresponds to a version of the ``average-distance problem'' studied in the calculus of variation and shape optimization community (\cite{BOS, BS03}).
	
	We investigate the theoretical properties satisfied by a  principal curve $f:[0,1]\to\R^d$ with length at most $L$ associated to a probability distribution with second-order moment. We suppose that the probability distribution is not supported on the image of a curve with length $L$. Studying open as well as closed optimal curves, we show  that they have finite curvature.		We also derive  a  first order Euler-Lagrange equation. This equation is then used to show that a  length-constrained principal curve in two dimension  has no multiple point. Finally, some  examples of optimal curves are presented.
	\end{abstract}



\section{Introduction}

\subsection{Context of the problem and motivation}

We focus on the problem:
	\begin{equation}
\begin{minipage}{0.9\textwidth}
find a  curve $f:[0,1]\to \R^d$ minimizing the quantity  $$
\E\left[d(X,\mbox{Im} f)^2\right]=\int d(x,\mbox{Im} f)^2d\mu(x),
$$
over all curves with length $\L(f)$, such that $\L(f)\leq L$.
\end{minipage}\label{eq:crit1}
\end{equation}
 Here, $d(\cdot,\cdot)$ is the Euclidean distance from a point
 to a set, $\mbox{Im} f$ is the image of $f$, and $X$ is some random vector with distribution $\mu$, taking its values in $\R^d$. As an illustration, two examples of length-constrained principal curves, fitted via a stochastic gradient descent algorithm, are presented in Figure \ref{fig:exPC}.
 \begin{figure}
 	\centering\includegraphics[width=0.5\textwidth]{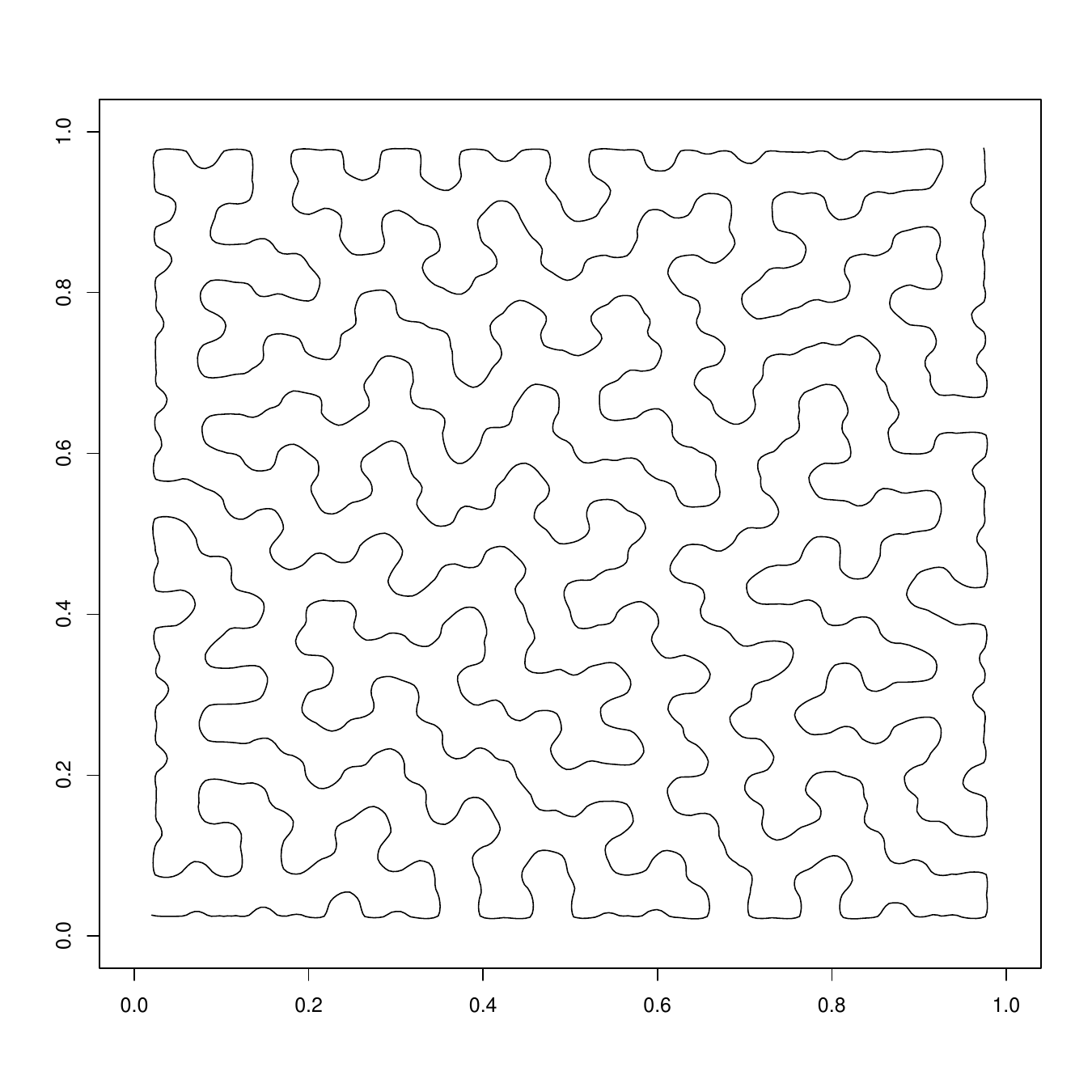}\hfill
 	\includegraphics[width=0.5\textwidth]{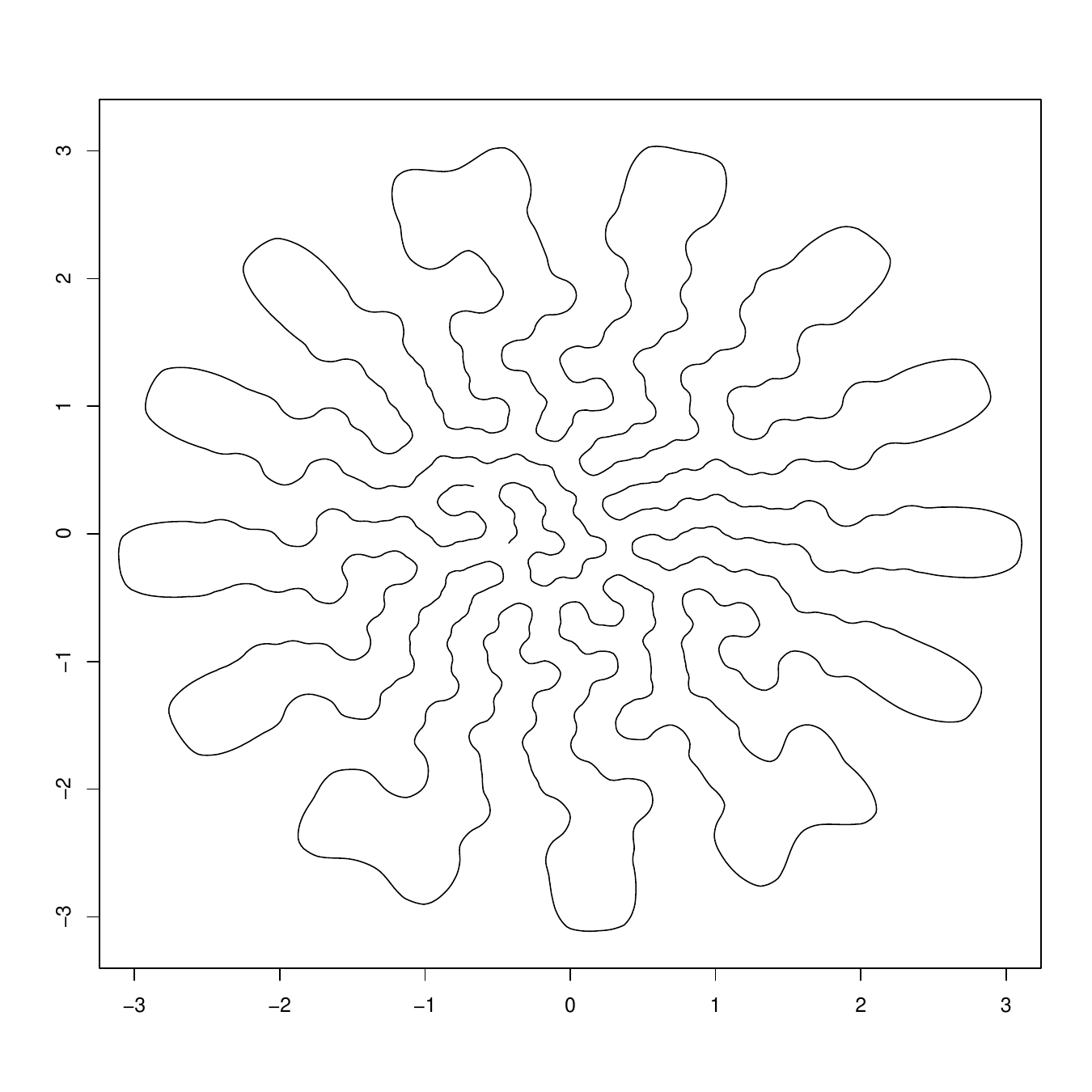}
 	\caption{Two examples of principal curves with length constraint: (a) Uniform distribution over the square $[0,1]^2$. (b) Standard Gaussian distribution.}\label{fig:exPC}
 \end{figure}
This corresponds to principal curves with length constraint, as described  by  \cite{KKLZ}. These authors show that there  exists indeed a minimizer  whenever $X$ is square integrable.
Observe that such a length constraint makes perfectly sense in the empirical case, that is in the statistical framework, when the random vector is replaced by a data cloud. Indeed, from a  practical point of view, it is essential to appropriately tune some parameter reflecting the complexity of the curve, in order to achieve a trade-off between a curve passing through all data points and a too rough one. The parameter selection issue was addressed in this statistical context for instance in \cite{BF}, \cite{AF} and \cite{GW}.

Originally, principal curves were introduced by \cite{HasStuet}, with a different definition,  based on the so-called self-consistency property.  In this point of view, a  curve $f$ is said to be self-consistent for a random  vector $X$ with finite second moment if it satisfies: $$f(t_f(X))=\E[X|t_f(X)]\quad \mbox{a.s.} ,$$ where the projection index $t_f$ is given by $$t_f(x)=\max\argmin_{t} \|x-f(t)\| .$$
The self-consistency property may be interpreted as follows: each point on the curve is the average of the mass of the probability distribution projecting there (for more details about the notion of self-consistency, see  \cite{TarFl}). Some regularity assumptions are made in addition: the  principal curve is required to be smooth ($C^\infty$), it does not intersect itself, and has finite length inside any ball in $\R^d$.
 The existence of principal curves designed	according to this definition cannot be proved in general (see \cite{DSextr96}, \cite{DSgeo96} for results obtained in the case of some particular distributions in two dimensions), which is the main  motivation for the least-square minimization definition proposed by \cite{KKLZ}.

Note that several other principal curve definitions, as well as algorithms, were proposed in the literature (\cite{Tib}, \cite{VVK}, \cite{Del},  \cite{Sankulk}, \cite{ETE}, \cite{OE},  \cite{GW}). Note also that principal curves, in their empirical version, have many applications in various areas (see for example \cite{HasStuet}, \cite{FO} for applications in physics, \cite{KK}, \cite{RN} in character and speech recognition, \cite{Brun}, \cite{SR}, \cite{BanRaf},  \cite{ETE, ETE1} in mapping and geology, \cite{De}, \cite{CAM}, \cite{ETE} in natural sciences, \cite{CCDH} in pharmacology, and \cite{WC},  \cite{DSD} in medicine, for the study of cardiovascular disease or cancer).

\subsection{Description of our results}

In this paper, we consider general distributions, assuming only that $X$ has a second order moment, and search for a   curve  which is optimal for problem  \eqref{eq:crit1}. We deal with open curves (with endpoints), as well as closed curves ($f(0)=f(1)$). Throughout, we will assume that the length-constraint is effective, that is the support of $X$ is not the image of a curve with length less than or equal to $L$. In this context, we prove that a minimizing curve cannot be self-consistent.
We also show that, for an optimal curve, the set of points with several different projections of the curve, called  ridge set in studies about the ``average-distance problem'' (see Section \ref{section:CompAna}), or ambiguity points in the principal curves literature, is negligible for the distribution of $X$.
 Then, we establish  that an optimal curve  is right- and left-differentiable everywhere and has bounded curvature.  Moreover, we obtain a first order Euler-Lagrange equation: we show that there exist $\lambda>0$ and a random variable $\hat{t}$ taking its values in $[0,1]$ such that  $\|X-f(\hat{t})\|=d(X,\mbox{Im} f)$ a.s. and 
\begin{equation}\label{eq:form1}
\E\left[ X-f(\hat{t})|\hat{t}=t\right]m_{\hat{t}}(dt)=-\lambda f ''(dt),
\end{equation} where $m_{\hat{t}}$ stands for the distribution of $\hat{t}$. To obtain that $\lambda \neq 0$, we use the fact that an optimal curve is not self-consistent.
Formula \eqref{eq:form1} allows us to propose in  dimension  $d=2$ a proof of the injectivity of an open principal curve as well as of a closed principal curve restricted to $[0,1)$.

\subsection{Comparison with previous results}\label{section:CompAna}

	Our framework is related to the  constrained  problem:
		\begin{equation}
	 \begin{minipage}{0.9\textwidth}
		minimize 
	$ \displaystyle \int_{\R^d}d(x,\Sigma)^pd\mu(x)$ over compact connected sets $\Sigma$ such that $\mathcal H^1(\Sigma)\leq L$.
	\end{minipage}\label{eq:irrigcons}
	\end{equation}

Here, $\mathcal H^\ell$ denotes $\ell$-dimensional Hausdorff measure. A connected question is the minimization of the penalized version of the criterion:
	 \begin{equation}
	\int_{\R^d}d(x,\Sigma)^pd\mu(x)+\lambda\mathcal H^1(\Sigma).\label{eq:irrigpen}
	\end{equation}

This issue, called in the calculus of variations and shape optimization community ``average-distance problem'' or, for $p=1$, ``irrigation problem'', has been  introduced by \cite{BOS,BS03} (see also the survey \cite{survAL}, and the references therein). Considering a compactly supported distribution, the penalized form is  studied for connected sets,  with $p=1$, in \cite{LuSle13}, and for curves, with $p\geq 1$, in \cite{LuSle}. In the first article, the authors prove that a minimizer is a tree made of a finite union of curves with finite length, and they provide a bound on the total curvature of these curves. In the second one, they show existence of a curve minimizing the penalized criterion
\begin{equation}\label{eq:lengthpen}
\int_{\R^d}d(x,\mbox{Im} f)^pd\mu(x)+\lambda\mathcal \L(f).
\end{equation} They give a bound on the curvature of the minimizer, and  prove that, in two dimensions, if  $p\geq 2$ or the distribution $\mu$ has a bounded density with respect to Lebesgue measure,  a minimizing curve is injective.

For the penalized irrigation problem \eqref{eq:irrigpen}, under the assumption that the distribution $\mu$, with compact support, does not charge the sets that have finite  $\mathcal H^{d-1}$ measure, which is true for instance if it has a density with respect to Lebesgue measure, an Euler-Lagrange equation is obtained for $p=1$ in \cite{BMS}, whereas \cite{LemenantR2} uses arguments involving  endpoints to derive one  in the case of the constrained version \eqref{eq:irrigcons}, in $\R^2$, under the same assumption on $\mu$. This assumption implies that $X$ is almost surely different from its projection on the curve, which is required for differentiability when $p= 1$, and, moreover, it is used to ensure negligibility of the ridge set.

For the constrained problem \eqref{eq:irrigcons}, if $\Sigma^*$ denotes a minimizer and  $\int_{\R^d}d(x,\Sigma)^pd\mu(x)>0$, it is shown in \cite{PaoSte} that $\mathcal H^1(\Sigma^*)=L$. A similar result in our context is stated in Corollary \ref{cor:L(f)} below.

Another related setting is the ``lazy travelling salesman problem'' of \cite{PoWo}: in $\R^2$, taking for $\mu$ an empirical distribution and considering closed curves, the authors study the penalized problem \eqref{eq:lengthpen} for $p=2$ (with $\lambda \L(f)$ replaced by  $\lambda \L^2(f)$). They show that for $\lambda$ large enough, the problem is reduced to a convex optimization.

Recall that we study in this manuscript the constrained problem \eqref{eq:crit1}, for open or closed  curves. In our context, the distribution of $X$ is not required to be compactly supported, and we do not need to assume that $\mu$ does not charge the sets with finite  $\mathcal H^{d-1}$ measure to derive an Euler-Lagrange equation. Indeed,  our proof does not rely on the fact that the ridge set is negligible. Besides, we prove that ambiguity points are actually negligible, which implies in particular that, for a given optimal curve, the Lagrange multiplier $\lambda$ in equation \eqref{eq:form1} only depends on the curve $f$.
We decided to focus on the case $p=2$
for which we can state the more complete  results.
In particular, we are only able to show the default of 
self-consistency of an optimal curve when  $p=2$. As already mentioned, this is a key point to get the main result.
Observe that it would be interesting to define a counterpart of the default of self-consistency when considering  other values of $p$.

\subsection{Organization of the paper}

Our document is organized as follows.
Section \ref{section:not} introduces relevant notation and recalls some basic facts about length-constrained principal curves. 
In Section \ref{section:res}, negligibility of ambiguity points is given in Proposition \ref{prop:hatXneg}, and the main result is stated in his complete form  in Theorem \ref{theo:main}.

Injectivity results are presented in \ref{section:inj}.
Finally, we give in Section \ref{section:exe} explicit examples of optimal curves.

\section{Definitions and notation}\label{section:not} 
For $d\geq 1$, the space $\R^d$ is equipped with the standard Euclidean norm, denoted by $\|\cdot\|$. The associated  inner product between two elements $u$ and $v$ is denoted by $\langle u,v\rangle$. Let $\mathcal H^1$ denotes the 1-dimensional Hausdorff measure in $\R^d$.

For $x\in\R^d$, $A\subset \R^d$, let $ d(x, A)=\inf_{y\in A}\|x-y\|$ denote the distance from point $x$ to set $A$.
For $r>0$, let $B(x,r)$ and $\bar B(x,r)$ denote, respectively, the open and the closed balls with center $x$ and radius $r$. Also, let $\partial A$ stand for the boundary of $A$, $\mbox{Card}(A)$ for its cardinality, and $\mbox{diam}(A)=\sup_{x,y\in A}\|x-y\|$  for its diameter.

For every $x\in \R^d$, let $x^j$ be its $j$-th component, for $j=1,\dots,d$, that is $x=(x^1,\dots,x^d)$. For every $x=(x^1,\dots,x^d)\in \R^d$, we set $\|x\|_\infty=\max_{1\leq i\leq d}|x^j|$.

 Let $(\Omega,\mathcal F, \P)$ be a probability space and $X$  a random vector on $(\Omega,\mathcal F, \P)$ with values in $\R^d$, such that $\E[\|X\|^2]<\infty$.
We will consider curves, that are continuous functions
\begin{align*}
f:[0,1]&\to\R^d\\
t &\mapsto (f^1(t),\dots,f^d(t)).
\end{align*} 
For such a curve $f:[0,1]\to\R^d$, let $\L(f)\in[0,\infty]$ denote its length, defined by
\begin{equation}
\L(f)=\sup\sum_{i=1}^n\|f(t_i)-f(t_{i-1})\|,\label{eq:length}
\end{equation}
where the supremum is taken over all possible subdivisions $0=t_0\leq \dots\leq t_n=1$, $ n\geq 1$ (see, e.g., 
\cite{AlexRes}). Let $\mbox{Im} f$ denote the image of $f$.

Let $$\Delta(f)=\E\left[d(X,\mbox{Im}f)^2\right],$$
and, for $L\geq 0$, $$G(L)=\min\{\Delta(f),f\in\mathcal C_L\},$$
where, in the sequel,  $\mathcal C_L$ will denote either one of the following sets of curves:  
\begin{align*}
&\{f\in[0,1]\to\R^d,\L(f)\leq L\},\\
&\{f\in[0,1]\to\R^d,\L(f)\leq L,f(0)=f(1)\}.
\end{align*}
Curves belonging to the latter set are closed curves. Note that $G$ is well-defined. Indeed, \cite{KKLZ} have shown the existence of an open curve $f$ with $\L(f)\leq L$ achieving the infimum of the criterion $\Delta(f)$, and the same proof applies for closed curves.

It will be useful to rewrite $ G(L)$, for every $L\geq 0$,
 as the minimum of the quantity $$\E[\|X-\hat X\|^2]$$ over all possible random vectors $\hat X$ taking their values in the image $\mbox{Im} f$ of a curve $f\in\mathcal C_L$.

\begin{rem}If $f:[0,1]\to\R^d$ is Lipschitz with constant $L$, its length is at most $L$. This follows directly from the definition of the length \eqref{eq:length}.
	Conversely, if the curve $f:[0,1]\to\R^d$ has length $\L(f)\leq L$, then there exists a curve with the same image which is Lipschitz with constant $L$. Indeed, a curve with finite length may be parameterized by arc-length (1-Lipschitz) (see, e.g., \citet[Theorem 2.1.4]{AlexRes}). 
\end{rem}

\begin{rem}
	Let $L\geq 0$. Suppose that $\hat{X}$ satisfies $G(L)=\E[\|X-\hat{X}]\|^2]$. Writing 
$$
\E[\|X-\hat X\|^2]=\E[\|X-\hat X-\E[X-\hat{X}]\|^2]+\|\E[X]-\E[\hat{X}]\|^2,
$$we see that, necessarily, \begin{equation}\label{eq:XhatX}
\E[X]=\E[\hat{X}],
\end{equation} since, otherwise, the criterion could be made strictly smaller by replacing $\hat{X}$ by the translated variable $\hat{X}+\E[X]-\E[\hat{X}]$, which contradicts the optimality of $\hat{X}$.

Observe that \eqref{eq:XhatX} remains true in a more general setting, as soon as the constraint corresponds to a quantity invariant by translation.

\end{rem}

\section{Main results and proofs}\label{section:res}

\subsection{Negligibility of the ridge set}

Given a curve $f :[0,1]\to \R^d$, consider the set $$\mathcal P_f(x)=\{y\in \mbox{Im} f,\|x-y\|=d(x, \mbox{Im} f)\}=\bar B(x,d(x,\mbox{Im} f))\cap \mbox{Im} f.$$ 
If $\mathcal P_f(x)$ has cardinality at least 2, $x$ is called an ambiguity point in the principal curves literature (see \cite{HasStuet}). Properties of the set of such points, named ridge set in the shape optimization community, have been studied for instance in \cite{ManMen}. In particular, the ridge set is measurable.  Using property \eqref{eq:XhatX}, it may be shown that the ridge set of an optimal curve for $X$ is negligible for the distribution of $X$. Section \ref{section:proof:hatXneg} below presents the proof of this result, as well as the proof of measurability, provided for the sake of completeness.

\begin{pro}\label{prop:hatXneg}\begin{enumerate}
		Let $f\in\mathcal C_L $ be an optimal curve for $X$ ($\Delta(f)=G(L)$).
		\item The set  $ \mathcal A_f=\{x\in\R^d,\mbox{Card}(\mathcal P_f(x))\geq 2\}$ of ambiguity points is measurable.
		\item The set $ \mathcal A_f$ is negligible for the distribution of $X$.
	\end{enumerate}

\end{pro}

\begin{rem}
The fact that the ridge set is negligible for the distribution of $X$  may be extended to the context of computing optimal trees under $\mathcal H^1$ constraint. Indeed, the result relies on property \eqref{eq:XhatX}, and $\mathcal H^1$ measure  is  translation invariant.
\end{rem}

\subsection{Main theorem and comments}
Recall that a signed measure on $(\Omega,\mathcal F)$ is a function $m:\mathcal F\to \R$ such that $m(\emptyset)=0$ and $m$ is $\sigma$-additive, that is $m\left(\bigcup_{k\geq 1}A_k\right)=\sum_{k\geq 1}m(A_k)$ for any sequence $(A_k)_{k\geq1}$ of pairwise disjoint sets.
For an $\R^d$-valued signed measure $m$ on $[0,1]$, that is $m=(m^1, \dots,m^d)$, where each $m^j$ is a signed measure, and for  $g:[0,1]\to\R^d$ a measurable function, we will use the following notation: $\int \langle g(t),m(dt)\rangle=\sum_{j=1}^d\int g^j(t)m^j(dt). $

A probability space $(\tilde{\Omega},\tilde{\mathcal F}, \tilde{\P})$ will be called an extension of $(\Omega,\mathcal F,\P)$ if there exists a random vector $\tilde{X}$ defined on $(\tilde{\Omega},\tilde{\mathcal F}, \tilde{\P})$, with  the same distribution $\mu$ as $X$. For simplicity, we still denote this random vector by $X$ throughout the paper.

\begin{theo}\label{theo:main}
	Let $L>0$ such that $G(L)>0$ and let $f\in\mathcal C_L$ such that $\Delta(f)=G(L)$. Then, $\L(f)= L$. Assuming that $f$ is  $L$-Lipschitz, we obtain that

\begin{itemize}	
		\item $f$ is right-differentiable on $[0,1)$, $\|f '_r(t)\|=L$ for all $t\in[0,1)$,
		\item $f$ is left-differentiable on $(0,1]$, $\|f '_\ell(t)\|=L$ for all $t\in (0,1]$,
	\end{itemize} and there exists a unique signed measure $f ''$ on $[0,1]$ (with values in $\R^d$) such that \begin{itemize}
	\item $f ''((s,t])=f'_r(t)-f'_r(s)$ for all $0\leq s\leq t<1$,
	\item $f ''([0,1])= 0$. 
\end{itemize}
In the case $\mathcal C_L=\{f:[0,1]\to \R^d, \L(f)\leq L\}$, we also have \begin{itemize}
	\item $f''(\{0\})=f'_r(0)$,
	\item $f''(\{1\})=-f'_\ell(1)$.
\end{itemize}

Moreover, there exists a unique $\lambda>0$ and,  there exists  a random variable $\hat{t}$ with values in $[0,1]$, defined on an extension $(\tilde{\Omega},\tilde{\mathcal F}, \tilde{\P})$ of the probability space $(\Omega,\mathcal F,\P)$, such that
\begin{itemize}
	\item $\|X-f(\hat{t})\|=d(X,\mbox{Im}f)$ a.s.,
	\item for every bounded Borel function $g:[0,1]\to \R^d$, 
	\begin{equation}\label{eq:formuleth}
	\E\left[\langle X-f(\hat{t}), g(\hat{t})\rangle\right]=-\lambda\int_{[0,1]}\langle g(t), f ''(dt)\rangle.
	\end{equation}
\end{itemize} 
\end{theo}

\begin{rem}
	Let $m_{\hat{t}|X}$ denote the conditional distribution of $\hat{t}$ given $X$. Then, equation \eqref{eq:formuleth} can be written in the following form: $$\int_{\R^d}\int_{[0,1]}\langle x-f(t),g(t)\rangle m_{\hat{t}|X}(x,dt) d\mu(x)=-\lambda\int_{[0,1]}\langle g(t), f ''(dt)\rangle.$$ 
\end{rem}

\begin{rem}\label{rem:formuleIP} Whenever the function  $g$ is absolutely continuous, an integration by parts (see for instance \citet[Theorem 21.67 $\&$ Remarks 21.68]{HewStr}) shows that equation \eqref{eq:formuleth}  may also be written
	\begin{equation}\label{eq:formuleIP}
	\E\left[\langle X-f(\hat{t}), g(\hat{t})\rangle\right]=\lambda\int_{0}^1\langle g'(t), f_r '(t)\rangle dt.
	\end{equation}
	To see this, let us write 
	$$ f ''([0,1])g(1)=f''(\{0\})g(0)+\int_{(0,1]}\langle g(t), f ''(dt)\rangle+\int_{(0,1]}\langle g'(s),f''([0,s])\rangle ds.$$
	Since $f''([0,1])=0$, we have $$0=\int_{[0,1]}\langle g(t), f ''(dt)\rangle+\int_{(0,1]}\langle g'(s), f_r'(s)\rangle ds,$$ which, combined with \eqref{eq:formuleth}, implies the announced formula \eqref{eq:formuleIP}.
\end{rem}

\begin{rem}If the curve $f$ has an angle at  $t$, which means that $f'_r(t)\neq f'_\ell(t)$, we see that 	$$\E[(X-f(\hat{t}))\I_{\{\hat{t}=t\}}]=-\lambda f''(\{t\})=\lambda(f_\ell'(t)-f_r'(t))\neq 0.$$ So, at an angle, $\P(\hat{t}=t)>0$.
	
	Besides, when $\mathcal C_L=\{f:[0,1]\to \R^d, \L(f)\leq L\}$, we have
	$$\E[(X-f(\hat{t}))\I_{\{\hat{t}=0\}}]=-\lambda f''(\{0\})=-\lambda f'_r(0),$$
	which cannot be zero, since $f'_r(0)$ has norm $L>0$. This implies that $\P(\hat{t}=0)>0.$

\end{rem}

\begin{rem}
	Regarding the random variable	$\hat t$,
	let us mention that $\hat{t}$ is unique almost surely whenever the curve is injective since $f(\hat t)$ is unique almost surely (it  is the case in dimension $d\leq 2$ ; see Section \ref{section:inj}). 
	 In general, it is worth pointing out that  Theorem \ref{theo:main} does not ensure that it is a function of $X$, as $(X,\hat{t})$ is, in fact, obtained as a limit in distribution of $(X,\hat{t}_n)$ for some sequence $(\hat{t}_n)_{n\geq 1}$.
	 	Besides, note that we do not know whether $\lambda$ depends on the curve $f$.

\end{rem}

\begin{rem}[Principal curves in dimension 1]

	Let $\mathcal C_L=\{f:[0,1]\to \R^d, \L(f)\leq L\}.$ 
	It may be of interest to consider the simplest case of dimension 1, where the problem may be solved entirely and explicitly
	Assume that  $X$ is a real-valued random variable, and that, for some length $L>0$, $G(L)>0$. Consider an optimal curve $f$ with length $\L(f)\leq L$. Using  Corollary \ref{cor:L(f)}  below, we have that, in fact, $\L(f)=L$, so that the image of $f$  is given by an interval $[a,a+L]$.
	In this context,
	solving directly the length-constrained principal curve problem in dimension 1 leads to minimizing in $a$ the quantity  $$\Delta(a):=\E\left[d(X,\mbox{Im}f)^2\right]=\E[(X-a)^2\I_{\{X<a\}}]+\E[(X-a-L)^2\I_{\{X>a+L\}}].$$ 
	The function $\Delta$ is differentiable in $a$, with derivative given by $$\Delta'(a)=2\E[(a-X)\I_{\{X<a\}}]+2\E[(a+L-X)\I_{\{X>a+L\}}].$$ Moreover, $\Delta'$  admits a right-derivative $\Delta''_r(a)=2(\P(X<a)+\P(X>a+L))$, which is positive since $G(L)>0$ implies that we do not have $X\in [a,a+L]$ almost surely.
	Hence, $\Delta$ is strictly convex, which shows that the minimizing $a$ is unique, so that the image of the principal curve $f$ is also uniquely defined.

	Besides, observe that equation \eqref{eq:formuleth} from Theorem \ref{theo:main} takes the following form in dimension 1: for every bounded Borel function $g:[0,1]\to \R^d$,  $$\E[(X-a)\I_{\{X<a\}}g(0)] +\E[(X-a-L)\I_{\{X>a+L\}}g(1)]=\lambda L (g(1)-g(0)).$$
In particular, we get
\begin{align*}
&\E[(X-a)\I_{\{X<a\}}] =-\lambda L ,\\
&\E[(X-a-L)\I_{\{X>a+L\}}]=\lambda L,
\end{align*}which characterizes $\lambda$.
Let us stress that we directly see in this case that $\lambda>0$, since, otherwise $X\in [a,a+L]$ almost surely, which contradicts the fact that $G(L)>0.$

\end{rem}

\subsection{Proof of Proposition \ref{prop:hatXneg}}\label{section:proof:hatXneg}	
\begin{enumerate}
		\item Note that \begin{align*}
			\mathcal A&=\{x\in\R^d,\mbox{Card}(\bar B(x,d(x,\mbox{Im} f))\cap \mbox{Im} f)\geq 2\}\\
			&=\{x\in\R^d,\mbox{diam}(\bar B(x,d(x,\mbox{Im} f))\cap \mbox{Im} f)> 0\}\\
			&=\R^d\setminus\{x\in\R^d,\mbox{diam}(\bar B(x,d(x,\mbox{Im} f))\cap \mbox{Im} f)= 0\}.
		\end{align*}
		For every $x\in \R^d$, we may write $$\mbox{diam}(\bar B(x,d(x,\mbox{Im} f))\cap \mbox{Im} f)=\lim_{n\to\infty} \mbox{diam}(B(x,d(x,\mbox{Im} f)+1/n)\cap \mbox{Im} f).$$
		Since $f$ is continuous, $f([0,1]\cap \mathbb Q)$ is 
		dense in $\mbox{Im} f$. For every $n\geq 1$, the countable set
		$B(x,d(x,\mbox{Im} f)+1/n)\cap f([0,1]\cap \mathbb Q)$ is dense in $B(x,d(x,\mbox{Im} f)+1/n)\cap \mbox{Im} f$, so that both sets have the same diameter.
		Yet, it can be easily checked that the diameter of a countable set	is measurable, and finally, we obtain that the set $\mathcal A$ of ambiguity points is measurable.
		\item To begin with, we prove that, for every $j=1,\dots,d$, it is possible to construct a random vector $\hat X$ with values in $\mbox{Im} f$ such that $\|X-\hat X\|=d(X,\mbox{Im}f)$ a.s., and $$\hat X^j=\max \pi_j(\bar B(X,d(X,\mbox{Im} f))\cap \mbox{Im} f).$$
			Here,   $\pi_j$ stands for the projection onto direction $j$, that is, for $x=(x^1,\dots,x^{d})\in \R^d$,  $\pi_j(x)=x^j$.
			Let $\{t_1,t_2,\dots\}$ be an enumeration of the countable set $[0,1]\cap \mathbb Q$. 	
			Let  $\e>0$, $x\in \R^d$. 	
			First, note that the set $\{t\in[0,1], \|f(t)-x\|< d(x, \mbox{Im} f)+\e\}$ is open. It is nonempty since the distance from $x$ to the closed set $\mbox{Im} f$ is attained. We deduce from this that $\mbox{Card}(\{t\in[0,1]\cap \mathbb Q, \|f(t)-x\|\leq d(x, \mbox{Im} f)+\e\})=\infty.$ 
			Let us define the sequence $(k_\e^n(x))_{m\in\N}$ by\begin{align*}
				&k_\e^1(x)=\min\{k: \|f(t_k)-x\|\leq d(x, \mbox{Im} f)+\e\}\\
				&k_\e^{m+1}(x)=\min\{k>k_\e^m(x): \|f(t_k)-x\|\leq d(x, \mbox{Im} f)+\e\},\quad m\in\N.
			\end{align*}	
			
			Let $j\in\{1,\dots,d\}$. We set	$$p^*(x)=\min\{p\geq 1, f^j(t_{k_\e^p(x)})\geq  \sup_{m\in\N}f^j(t_{k_\e^m(x)})-\e\}.$$  We define $\hat X_\e(x)=f(t_{k_\e^{p^*(x)}(x)})$, which is a measurable choice.
			Notice that, since $\{f^j(t_{k_\e^m(x)}),m\in\N\}=\pi_j(\bar B(x,d(x,\mbox{Im} f)+\e)\cap f([0,1]\cap\mathbb Q))$ is dense in $\pi_j(\bar B(x,d(x,\mbox{Im} f)+\e)\cap \mbox{Im} f)$, both sets have the same supremum.
			
			Let 
			$$\Pi_\e(x)=\pi_j(\bar B(x,d(x,\mbox{Im} f)+\e)\cap \mbox{Im} f),\quad \Pi(x)=\pi_j(\bar B(x,d(x,\mbox{Im} f))\cap \mbox{Im} f).$$
			The limit of $\hat X^j_\e(x)$ is given by $\lim_{\e\to 0}\max \Pi_\e(x)$. Yet, note that, for every $\e$, $\Pi(x)\subset \Pi_\e(x)$ so that \begin{equation}\label{eq:pi1}
				\max\Pi(x)\leq\max\Pi_\e(x).
			\end{equation} Moreover, if $\e$ is small enough, then for all $y\in\Pi_\e(x)$, $d(y, \Pi(x))\leq \eta(\e)$, where $\eta$ tends to 0 with $\e$, and, thus,  \begin{equation}\label{eq:pi2}\max\Pi_\e(x)\leq \max\Pi(x) + \eta(\e).\end{equation} Combining inequalities \eqref{eq:pi1} and \eqref{eq:pi2}, we obtain that $\lim_{\e\to 0}\max \Pi_\e(x)=\max \Pi(x)$.
			
			Set $\e_n=1/n$. Up to an extraction, we may assume that $(\hat X_{\e_n}(X),X)$ converges in distribution to $(\hat X,X)$ as $n\to \infty$. The random vector $\hat X$ satisfies $\|X-\hat X\|=d(X,\mbox{Im} f)$ and $\hat X^j=\max \Pi(X)$.
			
			Similarly, as may be seen by replacing $X$ by $-X$,  there exists 
			a random vector $\hat Y$ with values in $\mbox{Im} f$ such that $\|X-\hat Y\|=d(X,\mbox{Im}f)$ a.s., and   $$\hat Y^j=\min \pi_j(\bar B(X,d(X,\mbox{Im} f))\cap \mbox{Im} f).$$ 
			
			Now, we use this result to show that $\mathcal A$ is negligible for the distribution of $X$.	Assume that $\P(\mbox{Card}(\mathcal P_f(X))\geq 2)>0.$ There exists a first coordinate $j$ such that $\P(\mbox{Card}(\mathcal \pi_j(\mathcal P_f(X)))\geq 2)>0.$
			Then, it is possible to construct $\hat X^j$ and $\hat Y^j$ such that $\P(\hat X^j\geq\hat Y^j)=1$ and $\P(\hat X^j>\hat Y^j)>0$. Yet, by property \eqref{eq:XhatX}, $\E[\hat X]=\E[X]=\E[\hat Y]$, and, in particular, $\E[\hat X^j]=\E[\hat Y^j]$, which leads to a contradiction. Thus,  $\P(\mbox{Card}(\mathcal P_f(X))=1)=1.$

	\end{enumerate}

\bigskip

In the next sections, we present two lemmas, which are important both independently and  for obtaining the main result Theorem \ref{theo:main}.

\subsection{Properties of the function $G$}

The first lemma is about the monotonicity and continuity  properties of the function $G$.
Observe that $G$ is nonincreasing, since $\{f
:[0,1]\to \R^d,\L(f)\leq L_1\}\subset \{f:[0,1]\to \R^d,\L(f)\leq L_2\}$ when $L_1<L_2$, so that $G(L_2)\leq G(L_1)$.

\begin{lem}\label{lem:propG}\begin{enumerate}
		
		\item The function $G$ is continuous.
		\item The function $G$ is strictly decreasing over $[0,L_0)$, where $L_0=\inf\{L\geq 0,G(L)=0\}\in\R_+\cup\{\infty\}$.
		
	\end{enumerate}
	
\end{lem}

In particular, Lemma \ref{lem:propG} admits the next useful  corollary. \begin{cor}\label{cor:L(f)}
	For  $L>0$, if $G(L)>0$ and  $f\in\mathcal C_L$ is such that $\Delta(f)=G(L), $ then $\L(f)=L$. 

\end{cor}
\begin{proof}
	If $\L(f)<L$, then Lemma \ref{lem:propG} would imply  $G(\L(f))>G(L)=\Delta(f)$, which contradicts the definition of $G$.
\end{proof}

\begin{proof}[Proof of Lemma \ref{lem:propG}]
1.	Set $L\geq 0$. Let us show that $G$ is continuous at the point $L$. Let $(L_k)_{k\in \N}$ be a sequence in $\R_+$ converging to  $L$, with $L_k\neq L$ for all $k\in\N$. Let $f\in\mathcal C_L$ be such that  $\Delta(f)=G(L)$, and let $\hat X$ stands for a random vector taking its values in $\mbox{Im} f$ such that $\|X-\hat X\|=d(X,\mbox{Im}f)$ a.s. For every $k\in\N$, let $f_k:[0,1]\to\R^d$ be a curve such that $\L(f_k)\leq L_k$, $\Delta(f_k)=G(L_k)$ and $\|f_k(t)-f_k(t')\|\leq L_k |t-t'|$ for $t,t'\in[0,1]$.
		
		Observe that the sequence $(G(L_k))_{k\in\N}$ is bounded since $\E[\|X\|^2]<\infty$. Let us show that $G(L)$ is the unique limit point of this sequence. Let $\gamma:\N\to\N $ be any increasing function. Our purpose is to show that the sequence  $(G(L_{\gamma(k)}))_{k\in\N}$ converges to $G(L)$.
		
		Let us check that the  $f_k$ are equi-uniformly continuous and that the sequence $(f_k(0))$ is bounded. Since the sequence $(L_k)_{k\in\N}$ is bounded, say by $L'$, the $f_k$ are Lipschitz with common Lipschitz constant $L'$, and, thus, they are equi-uniformly continuous. For every $k\in\N$, $t\in[0,1]$, we have $\|f_k(t)\| \ge \|f_k(0)\| - L' t\ge \|f_k(0)\| - L'$. Thus, if there exists an increasing function $\kappa:\N\to \N$ such that $\|f_{\kappa(k)}(0)\| \to\infty$, one has $G(L_{\kappa(k)})\to\infty$, which is impossible since $G(L_k)\le \E[\|X\|^2]<\infty$. So, the sequence $(f_k(0))_{k\in\N}$ is bounded.

		Consequently, there exists an increasing function $\sigma:\N\to\N $ such that the subsequence $(f_{\sigma\circ\gamma(k)})_{k\in\N}$ converges uniformly to some function $\phi:[0,1]\to\R^d$.
		Note that the curve $\phi$ is $L$-Lipschitz, since for all $t,t'$, 
		\begin{align*}\|\phi(t)-\phi(t')\|&\leq 
			\|\phi(t)-f_{\sigma\circ\gamma(k)}(t)\|+\|f_{\sigma\circ\gamma(k)}
			(t)-f_{\sigma\circ\gamma(k)}(t')\|-\|f_{\sigma\circ\gamma(k)}(t')-\phi(t')\|\\
			&\leq 
			\|\phi(t)-f_{\sigma\circ\gamma(k)}(t)\|+L_{\sigma\circ\gamma(k)}|t-t'|-\|f_{\sigma\circ\gamma(k)}(t')-\phi(t')\|,
		\end{align*}which implies, taking the limit as $k\to \infty$,
		$\|\phi(t)-\phi(t')\|\leq L|t-t'|$.
		We have $\L(\phi)\leq \lim_{k\to\infty}L_k=L.$  
		Now, observe that
		\begin{align*}
			&\min_t\|X-f_{\sigma\circ\gamma(k)}(t)\|^2
			-\min_t\|X-\phi(t)\|^2\\&=
			\left(\min_t\|X-f_{\sigma\circ\gamma(k)}(t)\|
			-\min_t\|X-\phi(t)\|\right)\left(\min_t\|X-f_{\sigma\circ\gamma(k)}(t)\|
			+\min_t\|X-\phi(t)\|\right)
			\\&\leq
			\|\phi(t^*)-f_{\sigma\circ\gamma(k)}(t^*)\|(
			\|X-f_{\sigma\circ\gamma(k)}(t^*)\|+\|X-\phi(t^*)\|),
		\end{align*}
		where $\|X-\phi(t^*)\|=\min_t\|X-\phi(t)\|$. 
		Since $\E[\|X\|^2]<\infty$ and $f_{\sigma\circ\gamma(k)}$ converges uniformly to $\phi$,
		this shows that  $\Delta(f_{\sigma\circ\gamma(k)})$ converges to $\Delta(\phi)$.
		
		Finally, let us check that $\Delta(\phi)=G(L)$. If $L=0$, then for every $k$, $L_k\geq L$, thus $\Delta(f_{\sigma\circ\gamma(k)})=G(f_{\sigma\circ\gamma(k)})\leq G(0)$ for every $k$. Consequently, $\Delta(\phi)\leq G(0)$, which implies $\Delta(\phi)= G(0)$ since $\phi$ has length 0. If $L>0$, note that, for every $k$,  $\frac {L_k}{L}\hat X$ is a random vector with values in $\frac {L_k}{L}\mbox{Im} f$ since $\hat X$ is taking its values in $\mbox{Im} f$. Moreover, $\frac {L_k}{L}f$ has length at most $L_k$    since $f$ has length $L$. Thus, for every $k$, $$\E\left[\left\|X-\frac{L_{\sigma\circ\gamma(k)}}L\hat X\right\|^2\right]\geq G(L_{\sigma\circ\gamma(k)})=\Delta(f_{\sigma\circ\gamma(k)}).$$ taking the limit as $k\to \infty$, we obtain  $$\E\left[\|X-\hat X\|^2\right]\geq \Delta(\phi),$$ which means that $\Delta(\phi)=G(L)$ since $\L(\phi)\leq L$.
		
2. We have to show that $G$ is strictly decreasing as long as the length constraint is effective (that is $G(L)> 0$).
		Let us prove that for $0\leq L_1<L_2$, we have $G(L_2)<G(L_1)$ if $G(L_1)>0$.
		Let $f:[0,1]\to \R^d$ such that $\L(f)\leq L_1$ and $\Delta(f)=G(L_1)$.
		For $t_0\in[0,1]$ and $r>0$, we define $\hat Z_{t_0,r}$ by
		$$\begin{cases}
		\hat Z^J_{t_0,r}=f^J(t_0)+r\land (X^J-f^J(t_0))\I_{\{X^J\geq f^J(t_0)\}}+(-r)\lor (X^J-f^J(t_0))\I_{\{X^J< f^J(t_0)\}},\\\mbox{where }J=\min\{i:|X^i-f^i(t_0)|=\|X-f(t_0)\|_{\infty}\}\\
		\hat Z^i_{t_0,r}=f^i(t_0) \mbox{ if }i\neq J, i=1,\dots,d.
		\end{cases}$$
		Observe that  $\hat Z_{t_0,r}$ takes its values in $$\mathcal C (t_0,r)=\bigcup_{j=1}^d\{x\in\R^d: x^i=f^i(t_0) \mbox{ for } i\neq j, |x^j-f^j(t_0)|\leq r\}.$$
		Indeed, all coordinates of $\hat Z_{t_0,r}$ are equal to the corresponding coordinate of $f(t_0)$ apart from the $J$-th  coordinate, that is the first coordinate  for which the distance between $X$ and $f(t_0)$ is the largest one. Let us check that $|\hat Z_{t_0,r}^J-f^J(t_0)|\leq r$.
		
		If $X^J\geq f^J(t_0)$, either $\hat Z_{t_0,r}^J-f^J(t_0)=r $, or $\hat Z_{t_0,r}^J-f^J(t_0)=X^J-f^J(t_0)\leq r$.
		
		If $X^J< f^J(t_0)$, either $f^J(t_0)-\hat Z_{t_0,r}^J=r $, or $f^J(t_0)-\hat Z_{t_0,r}^J=f^J(t_0)-X^J\leq r$.
		
		\begin{figure}[H]\centering
			\begin{tikzpicture}[scale=0.3]
			\draw[color=red] (0,-5) -- (0,11);
			\draw[color=red] (-8,3) -- (8,3);
			\draw  plot[smooth, tension=.7] coordinates {(11,16) (9.6,11.4) (7.2,7.8) (4,5) (0,3) (-4,2)(-9,1)(-15,0)};
			\draw[<->,color=vert1] (-8,4) -- (-0.2,4);
			\node at (-4.2,4.5) {$r$};
			\node at (1.5,2) {$f(t_0)$};
			\node at (0,3) {$\bullet$};
			\end{tikzpicture}
			\caption{Example in $\R^2$, illustrating the support of $\hat Z_{t_0,r}$.}
		\end{figure}
		
		Then, letting again $\hat X$ be a random vector with values in $\mbox{Im} f$ such  that $\|X-\hat X\|=d(X,\mbox{Im}f)$ a.s., we set
		$$\hat X_{t_0,r}=\hat X\I_{\{\|X-\hat X\|\leq \|X-\hat Z_{t_0,r}\|\}}+\hat Z_{t_0,r}\I_{\{\|X-\hat X\|>\|X-\hat Z_{t_0,r}\|\}}.$$
		Since $\|X-\hat Z_{t_0,r}\|^2=\|X-f(t_0)\|^2-\|X-f(t_0)\|_\infty^2+(\|X-f(t_0)\|_\infty-r)^2_{+}, $

		\begin{align*}
			&\|X-\hat X\|^2-\|X-\hat{X}_{t_0,r}\|^2\\&=\left[\|X-\hat X\|^2-\|X-\hat{Z}_{t_0,r}\|^2\right]_+\\
			&=\left[\|X-\hat X\|^2-\|X-f(t_0)\|^2+\|X-f(t_0)\|_\infty^2-(\|X-f(t_0)\|_\infty-r)^2_{+}\right]_+\\
			&\geq \left[\|X-\hat X\|^2-\|X-f(t_0)\|^2+\|X-f(t_0)\|_\infty^2-(\|X-f(t_0)\|_\infty-r)^2\right]_+\\
			&= \left[\|X-\hat X\|^2-\|X-f(t_0)\|^2+2r\|X-f(t_0)\|_\infty-r^2\right]_+\\
			&= \left[\|f(t_0)-\hat X\|^2+2\langle X-f(t_0),f(t_0)-\hat X\rangle+2r\|X-f(t_0)\|_\infty-r^2\right]_+\\
			&= \left[-\|f(t_0)-\hat X\|^2+2\langle X-\hat X,f(t_0)-\hat X\rangle+2r\|X-f(t_0)\|_\infty-r^2\right]_+\\
			&\geq \left[-\|f(t_0)-\hat X\|^2+2\langle X-\hat X,f(t_0)-\hat X\rangle+\frac{2r}{\sqrt{d}}\|X-f(t_0)\|-r^2\right]_+\\&\quad\mbox{ since for every }x\in\R^d, \|x\|\leq\sqrt{d}\|x\|_{\infty}\\
			&\geq \left[-\|f(t_0)-\hat X\|^2+2\langle X-\hat X,f(t_0)-\hat X\rangle+\frac{2r}{\sqrt{d}}\|X-\hat{X}\|-r^2\right]_+\\&\quad\mbox{ since }\|X-\hat{X}\|\leq\|X-f(t_0)\|. 
		\end{align*}
		
		Besides, $\hat X_{t_0,r}$ takes its values in $\mbox{Im} f\cup \C(t_0,r)$, which is the image of a  curve with length at most $L_1+4dr$, so that $ \E[\|X-\hat X_{t_0,r}\|^2]\geq G(L_1+4dr)$. 
		
		Thus, \begin{multline}\label{eq:minorGL1}G(L_1)\geq  G(L_1+4dr)\\+\E\left[\left[-\|f(t_0)-\hat X\|^2+2\langle X-\hat X,f(t_0)-\hat X\rangle+\frac{2r}{\sqrt{d}}\|X-\hat{X}\|-r^2\right]_+\right].
		\end{multline}
		
		Since $G(L_1)>0$, $\P(\|X-\hat X\|>0)>0$, thus there exist $\delta>0$ and $K<\infty$ such that $\eta:=P(K\geq\|X-\hat X\|\geq\delta)>0$.
		
		Recall that, for all $(t,t')$, we have $\|f(t)-f(t')\|\leq L_1|t-t'|$. Then, for every $p\geq 1,$ there exists $k$, $1\leq k\leq p$, such that $\|\hat X-f(\frac{k}{p})\|\leq\frac{L_1}{p}$ and so, we have $$\sum_{k=1}^{p}\I_{\left\{\|\hat X-f(\frac{k}{p})\|\leq\frac{L_1}{p}\right\}} \geq 1.$$ Thus,
		$$\sum_{k=1}^{p}\P\left(K\geq\|X-\hat X\| \geq \delta,\left\|\hat X-f\left(\frac{k}{p}\right)\right\|\leq\frac{L_1}{p}\right) \geq \eta.$$ Consequently, for every $p\geq 1$, there exists $t_p\in[0,1]$ such that $$\P\left(K\geq\|X-\hat X\| \geq \delta,\|\hat X-f(t_p)\|\leq\frac{L_1}{p}\right) \geq \frac{\eta}{p}>0.$$
		
		According to \eqref{eq:minorGL1}, we obtain 
		\begin{align*}
			G(L_1)&\geq  G(L_1+4dr)+\E\left[-\|f(t_p)-\hat X\|^2+2\langle X-\hat X,f(t_p)-\hat X\rangle+\frac{2r}{\sqrt{d}}\|X-\hat{X}\|-r^2\right]_+\\
			&\geq  G(L_1+4dr)+\E\left[\I_{\left\{K\geq\|X-\hat X\| \geq \delta,\left\|\hat X-f\left(t_p\right)\right\|\leq\frac{L_1}{p}\right\}}\left(-\frac{L_1^2}{p^2}-\frac{2KL_1}{p}+\frac{2r\delta}{\sqrt{d}}-r^2\right)\right]\\
			&\geq G(L_1+4dr)+\frac{\eta}{p}\left(-\frac{L_1^2}{p^2}-\frac{2KL_1}{p}+\frac{2r\delta}{\sqrt{d}}-r^2\right).
		\end{align*}
		Now, choosing $r>0$ such that $\frac{2r\delta}{\sqrt{d}}-r^2>0$ and $L_1+4dr\leq L_2$, we finally obtain, taking $p$  large enough, $$G(L_1)>G(L_1+4dr)\geq G(L_2).$$
	
\end{proof}

\subsection{Default of self-consistency}

The next lemma
 states that a  principal curve with length $\leq L$ does not satisfy the so-called  self-consistency property, provided that the constraint is effective, that is $G(L)>0$.
 
\begin{lem}\label{lem:noneq}
	Let $L>0$ such that $G(L)>0$, and let $f\in\mathcal C_L$ be such that $\Delta(f)=G(L).$ If $\hat{X}$ is a random vector with values in $\mbox{Im} f$ such that $\|X-\hat{X}\|=d(X,\mbox{Im}f)$ a.s., then $\P(\E[X|\hat{X}]\neq \hat{X})>0.$
\end{lem}
\begin{proof}
	First of all, observe that $\L(f)=L$ since $G(L)>0$, according to Corollary \ref{cor:L(f)}.
	Assume that $\E[X|\hat{X}]=\hat{X}$ a.s..
	
For $\e\in[0,1], $ we set $\hat{X}_\e=(1-\e)\hat{X}$. Then,
		$$\|X-\hat{X}_\e\|^2=\|X-\hat{X}+
		\e\hat{X}\|^2=\|X-\hat{X}\|^2+\e^2\|\hat{X}\|^2+2\e\langle X-\hat{X},\hat{X}\rangle.$$ Since $\E[X|\hat{X}]=\hat{X}$ a.s., 
		$\E[X-\hat{X}|\hat{X}]=0$ a.s., and thus, $\E[\langle X-\hat{X},\hat{X}\rangle]=\E[\langle \E[X-\hat{X}|\hat{X}],\hat{X}\rangle]=0,$ so that \begin{equation}\label{eq:decriteps}
			\E[\|X-\hat{X}_\e\|^2]=\E[\|X-\hat{X}\|^2]+\e^2\E[\|\hat{X}\|^2].
		\end{equation}
		The random vector $\hat{X}_\e$ is taking its values in the image of $(1-\e)f$, which has length $(1-\e)L$.
		Observe that \begin{equation}\label{eq:hatXfini}
			\E[\|\hat{X}\|^2]<\infty,
		\end{equation}since $\E[\|X\|^2]<\infty$ and \begin{align*}
			\E[\|\hat{X}\|^2]&\leq 2\E[\|X-\hat{X}\|^2]+2\E[\|X\|^2]\\&\leq  2\E[\|X-f(0)\|^2]+2\E[\|X\|^2]\\&\leq 6\E[\|X\|^2]+4\|f(0)\|^2.
		\end{align*}
		We will show that, adding to $(1-\e)f$ a curve with length $\e L$, it is possible to  build $\hat{Y}_\e$ with $\E[\|X-\hat{Y}_\e\|^2]<\E[\|X-\hat{X}\|^2]$, 
		which contradicts the optimality of $f$.
	
	For $\e\in[0,1]$, let $f_\e=(1-\e)f.$ We then define $\hat{X}_{\e,t_0,r}$ as the variable $\hat{X}_{t_0,r}$ corresponding to $f_\e$. More precisely, similarly to the proof of Lemma \ref{lem:propG}, we define, for  $t_0\in[0,1]$ and $r>0$, the random  vector $\hat{Z}_{\e,t_0,r}$, with values in 
		$$\mathcal C (t_0,r)=\bigcup_{j=1}^d\{x\in\R^d: x^i=f_\e^i(t_0) \mbox{ for } i\neq j, |x^j-f^j_\e(t_0)|\leq r\},$$ by 
		$$\begin{cases}
		\hat Z^J_{\e,t_0,r}=f_\e^J(t_0)+r\land (X^J-f^J_\e(t_0))\I_{\{X^J\geq f^J_\e(t_0)\}}+(-r)\lor (X^J-f^J_\e(t_0))\I_{\{X^J< f^J_\e(t_0)\}},\\\mbox{where }J=\min\{i:|X^i-f^i_\e(t_0)|=\|X-f_\e(t_0)\|_{\infty}\}\\
		\hat Z^i_{\e,t_0,r}=f^i_\e(t_0) \mbox{ if }i\neq J, i=1,\dots,d.
		\end{cases}$$
		We set $$\hat X_{\e,t_0,r}=\hat X\I_{\{\|X-\hat X_\e\|\leq \|X-\hat Z_{\e,t_0,r}\|\}}+\hat Z_{\e,t_0,r}\I_{\{\|X-\hat X_\e\|>\|X-\hat Z_{\e,t_0,r}\|\}}.$$
		By the same calculation as in the proof of Lemma \ref{lem:propG}, we obtain 
		$$
		\|X-\hat X_\e\|^2-\|X-\hat{X}_{\e,t_0,r}\|^2\geq\left[-\|f_\e(t_0)-\hat X_\e\|^2+2\langle X-\hat X_\e,f_\e(t_0)-\hat X_\e\rangle+\frac{2r}{\sqrt{d}}\|X-f_\e(t_0)\|-r^2\right]_+. $$
		Since $\|X-f_{\e}(t_0)\|\geq \|X-f(t_0)\|-\e\|f(t_0)\|\geq \|X-\hat{X}\|-\e\|f(t_0)\|,$ we get
		\begin{multline*}
			\|X-\hat X_\e\|^2-\|X-\hat{X}_{\e,t_0,r}\|^2\geq \left[-(1-\e)^2\|f(t_0)-\hat X\|^2+2(1-\e)\langle X-\hat X_\e,f(t_0)-\hat X\rangle\right.\\\left.+\frac{2r}{\sqrt{d}}\|X-\hat{X}\|-\frac{2r}{\sqrt{d}}\e\|f(t_0)\|-r^2\right]_+.
		\end{multline*}
		Thus, \begin{align}\label{eq:minorlem2}
			&\E\left[\|X-\hat X_\e\|^2-\|X-\hat{X}_{\e,t_0,r}\|^2\Big| \hat{X}\right]\nonumber\\&\geq
			\left[-\|f(t_0)-\hat X\|^2+2(1-\e)\langle \E[X| \hat{X}]-\hat X_\e,f(t_0)-\hat X\rangle+\frac{2r}{\sqrt{d}}\E\left[\|X-\hat{X}\|\Big| \hat{X}\right]-\frac{2r}{\sqrt{d}}\e\|f(t_0)\|-r^2\right]_+\nonumber\\&=
			\left[-\|f(t_0)-\hat X\|^2+2(1-\e)\langle \e\hat{X},f(t_0)-\hat X\rangle+\frac{2r}{\sqrt{d}}\E\left[\|X-\hat{X}\|\Big| \hat{X}\right]-\frac{2r}{\sqrt{d}}\e\|f(t_0)\|-r^2\right]_+\nonumber\\&\geq 
			\left[-\|f(t_0)-\hat X\|^2-2\e\|\hat{X}\|\|f(t_0)-\hat X\|+\frac{2r}{\sqrt{d}}\E\left[\|X-\hat{X}\|\Big| \hat{X}\right]-\frac{2r}{\sqrt{d}}\e\|f(t_0)\|-r^2\right]_+. 
		\end{align}
		Besides, since $G(L)>0$, there exist $\delta>0$, $K<\infty,$ such that
		$$\eta=\P\left(\|\hat{X}\|\leq K, \E\left[\|X-\hat{X}\|\Big| \hat{X}\right]\geq \delta\right)>0.$$
		Moreover, for every $p\geq 1$, $\sum_{k=1}^p \I_{\{\|\hat{X}-f\left(\frac{k}{p}\right)\|\leq \frac{L}{p}\}}\geq 1$ since $f$ is $L$-Lipschitz. Consequently, $$ \sum_{k=1}^{p}\P\left(\|\hat{X}\|\leq K, \E\left[\|X-\hat{X}\|\Big | \hat{X}\right]\geq \delta,\left\|\hat{X}-f\left(\frac{k}{p}\right)\right\|\leq \frac{L}{p}\right)\geq \eta.$$ Hence, setting $$A_p=\left\{\|\hat{X}\|\leq K, \E\left[\|X-\hat{X}\|\Big| \hat{X}\right]\geq \delta,\left\|\hat{X}-f\left(\frac{k}{p}\right)\right\|\leq \frac{L}{p}\right\},$$ we see that there exists $t_p\in[0,1]$ such that $\P(A_p)\geq\frac{\eta}{p}$. From \eqref{eq:minorlem2}, we get
		\begin{align*}
			\E&\left[\|X-\hat X_\e\|^2-\|X-\hat{X}_{\e,t_p,r}\|^2\right]\\&\geq 
			\E\left[\I_{A_p}	\left[-\|f(t_p)-\hat X\|^2-2\e\|\hat{X}\|\|f(t_p)-\hat X\|+\frac{2r}{\sqrt{d}}\E\left[\|X-\hat{X}\|\Big| \hat{X}\right]-\frac{2r}{\sqrt{d}}\e\|f(t_p)\|-r^2\right]_+\right]\\&\geq
			\P(A_p)\left[-\frac{L^2}{p^2}-\frac{2\e K L}{p}+\frac{2r\delta}{\sqrt{d}}-\frac{2r\e M}{\sqrt{d}}-r^2\right],
		\end{align*}where $M=\sup_{t\in[0,1]}\|f(t)\|$.
		Since $\hat{X}_{\e,t_p,r}$ takes its values in $f_\e([0,1])\cup \mathcal C(\e,t_p,r)$, which is the image of a curve with length at most $(1-\e)L+4dr,$ then choosing $r$ such that $4dr=\e L$, we have
		\begin{align*}
			\E\left[\left\|X-\hat{X}_{\e,t_p,\frac{\e L}{4d}}\right\|^2\right]&\leq 
			\E\left[\|X-\hat X_\e\|^2\right]-\frac{\eta}{p}\left(-\frac{L^2}{p^2}-\frac{2 K L\e}{p}+\frac{ L\delta\e}{2d^{3/2}}-\frac{ML\e^2}{2d^{3/2}}-\frac{L^2\e^2}{16d^2}\right)\\&=
			\E[\|X-\hat{X}
			\|^2]+\e^2\E[\|\hat{X}\|^2]+\frac{\eta L^2}{p^3}+\frac{2\eta K L\e}{p^2}-\frac{ \eta L\delta\e}{2d^{3/2}p}+\frac{\eta ML\e^2}{2d^{3/2}p}-\frac{\eta L^2\e^2}{16d^2p},
		\end{align*}using \eqref{eq:decriteps}. Then, taking $\e=\frac{\rho}{p}$, we get 
		$$\E\left[\left\|X-\hat{X}_{\frac{\rho}{p},t_p,\frac{\rho L}{4dp}}\right\|^2\right]\leq 
		\E[\|X-\hat{X}
		\|^2]+\frac{\rho^2}{p^2}\E[\|\hat{X}\|^2]+\frac{\eta L^2}{p^3}+\frac{2\eta K L\rho}{p^3}-\frac{ \eta L\delta\rho}{2d^{3/2}p^2}+\frac{\eta ML\rho^2}{2d^{3/2}p^3}-\frac{\eta L^2\rho^2}{16d^2p^3}.$$
		If $\rho$ is small enough, then ${\rho^2}\E[\|\hat{X}\|^2]-\frac{ \eta L\delta\rho}{2d^{3/2}}<0$. Then, taking $p$ large enough, this leads to a random  vector $\hat{Y}$, with values in the image of a curve with length at most $L$, such that  $\E[\|X-\hat{Y}\|^2]<\E[\|X-\hat{X}\|^2]$.

\end{proof}

Equipped with lemmas \ref{lem:propG} and \ref{lem:noneq}, we can present the proof of the main result.

\subsection{Proof of Theorem \ref{theo:main}}

To obtain a length-constrained principal curve, we have to minimize a function  which may not be differentiable. We propose to  build a discrete approximation of the principal curve $f$, using a chain of points $v^n_1,\dots,v^n_n$, $n\geq 1$, in $\R^d$. For every $n\geq 1$, linking the points yields a polygonal curve $f_n$. The properties of the principal curve $f$ will be shown by passing to the limit.  The chain of points is obtained by minimizing a $k$-means-like  criterion, which is differentiable, under a length-constraint This criterion is based on the distances from the random vector $X$ to the $n$ points and not to the corresponding segments of the polygonal line $f_n$, which allows to simplify the computation of the gradients.

We have chosen to present the proof for open curves, that is in the case $\mathcal C_L=\{\phi:[0,1]\to \R^d, \L(\phi)\leq L\}$. It adapts straightforwardly to the case of closed curves, which turns out to be even simpler since there are no endpoints and so all points of the curve play the same role. Note that the normalization factor  ``$n-1$'' below becomes ``$n$'' in the closed curve context.

\subsubsection*{First insight into the proof}To facilitate understanding, we sketch the proof in a simpler case. Assume that $X$ has a density with respect to Lebesgue measure, and consider a polygonal line $f_n$ with vertices $v_1^n,\dots,v_n^n$ obtained by minimizing under length constraint the  criterion
\begin{equation}
F_n^0(x_1,\dots,x_n)= \E\left[\min_{1\leq i\leq n}\|X-x_i\|^2\right].\label{eq:crit}
\end{equation}
For $h=(h_1,\dots,h_n)\in(\R^d)^n$, $\nabla F_n^0.h=\sum_{i=1}^{n}\E\left[-2\langle X-\hat{X}_n,h_i\rangle\I_{\{\hat{X}=v_i^n\}}\right]$, where $\hat X$ is such that $\|X-\hat{X}\|=\min_{1\leq j\leq n}\|X-v_i^n\|$.
For differentiability, it is convenient to write the length constraint as follows: $$(n-1)\sum_{i=2}^{n}\|x_i-x_{i-1}\|^2\leq L^2.$$
 Let $\hat{t}_n$ be defined by  $\hat{t}_n=\frac{i-1}{n-1}$ on the event $\{\hat X=v_i^n\}$. For a test function $g$, set $h_i=g\big(\frac{i-1}{n-1}\big)$ for $i=1,\dots,n$.
Then, we obtain the Euler-Lagrange equation 
\begin{equation}\label{eq:euldis}
\E\left[\langle X-f_n(\hat{t}_n),g(\hat{t}_n)\rangle\right]=-\lambda_n\int_{[0,1]}\langle g(t),f_n''(dt)\rangle.
\end{equation}

Up to an extraction, $f_n$ converges uniformly to an optimal curve and $\hat{t}_n$ converges in distribution. Using the default of self-consistency \eqref{lem:noneq}, it may be shown that every limit point of the sequence $(\lambda_n)_{n\geq 1}$ is positive. Together with the discrete Euler-Lagrange equation \eqref{eq:euldis}, this allows to prove  that $f_n''$ converges weakly to a signed measure $f''$. Finally, the desired Euler-Lagrange equation is obtained as the limit of \eqref{eq:euldis}.

\subsubsection*{Complete proof}
Let us now start the complete proof of the theorem. 

First, some notation is in order.
	Let $Z$ be a standard $d$-dimensional Gaussian vector, independent of $X$. 
Let $(\zeta_n)$, $(\eta_n)$ and $(\e_n)$ be sequences of positive real numbers such that
$$ \zeta_n= \mathcal O(1/n), \quad \eta_n=\mathcal O(1/n),\quad
n\e_n\to\infty, \quad \e_n\to 0.$$
We also introduce i.i.d. random vectors $\xi_1^n,\dots,\xi_n^n$, independent of $X$ and $Z$,  with same distribution as a centered random  vector $\xi$ with compactly supported density, such that $\|\xi\|\leq \eta_n$.

We will construct a sequence of polygonal lines converging to the optimal curve $f$ by linking points $v^n_1,\dots,v^n_n$ obtained by minimization of a criterion generalizing \ref{eq:crit}. Proving  differentiability in this case is a little more involved in this case.

To begin with, since the random vector $X$ is not assumed to have a density with respect to Lebesgue measure, we convolve it with a  Gaussian random vector: we define, for $n\geq 1$,  $X_n=X+\zeta_n Z$. So, $X$ is approximated by a sequence $(X_n)_{n\geq 1}$ of continuous  random variables. 

For $1\leq i\leq n,$ let $$t_i^n:=\dfrac{i-1}{n-1}.$$
In order to be able to prove results which are true for any optimal curve $f$, we have to ensure that the points $v^n_1,\dots,v^n_n$, $n\geq 1$ are located on this curve $f$. To this aim, we add to the criterion \ref{eq:crit} a penalty proportional to $$\sum_{i=1}^n\|x_i-f(t_i^n)\|^2.$$ With this penalty, we cannot affirm any more that the $x_i$'s are pairwise distinct. 
To overcome this difficulty, a random vector $\xi^n_i$ is added to each $x_i$: the points $x_i+\xi^n_i$, approximating the $x_i$'s, are almost surely pairwise distinct.

The desired  chain of points is then defined, for $n\geq 1$, by minimizing in $x=(x_1,\dots,x_n)\in(\R^d)^n$ the criterion	\begin{equation}
 F_n(x_1,\dots,x_n)= \E\left[\min_{1\leq i\leq n}\|X_n-x_i-\xi_i^n\|^2\right]+\e_n\sum_{i=1}^n\|x_i-f(t_i^n)\|^2,\label{eq:crit-tot}
\end{equation}
under the constraint \begin{equation}\label{eq:constLvi}
(n-1)\sum_{i=2}^n\|x_i-x_{i-1}\|^2\leq L^2.
\end{equation}

\begin{lem}\label{lem:existvi}
	There exists $(v^n_1,\dots,v^n_n)\in(\R^d)^ n$, satisfying 
$$(n-1)\sum_{i=2}^n\|v^n_i-v^n_{i-1}\|^2\leq L^2,$$
 such that  $$F_n(v^n_1,\dots,v^n_n)=\min\bigg\{F_n(x_1,\dots,x_n); (n-1)\sum_{i=2}^n\|x_i-x_{i-1}\|^2\leq L^2\bigg\}.$$
\end{lem}	
Let  $\hat{X}_n^x$ be such that $\hat{X}_n^x\in\{x_1+\xi_1^n,\dots,x_n+\xi_n^n\}$ and 
\begin{equation}\label{eq:hatXnx}
\|X_n-\hat{X}_n^x\|=\min_{1\leq i\leq n}\|X_n-x_i-\xi_i^n\|
\end{equation} almost surely.
In the sequel,  $\hat{X}_n$ will stand for $\hat{X}_n^{(v^n_1,\dots,v^n_n)}$.

\begin{lem}\label{lem:F-fini}
\begin{equation*}\sup_{n\geq 1}F_n(v^n_1,\dots,v^n_n)<\infty.\end{equation*} 
\end{lem}

	We define the sequence $(f_n)_{n\geq 1}$ of polygonal lines approximating $f$, where each $f_n:[0,1]\to \R^d$, $n\geq 1$, is given by $$f_n(t)=v^n_i+(n-1)\left(t-t_i^n\right)(v^n_{i+1}-v^n_i), \quad t^n_i\leq t\leq t^n_{i+1},\quad 1\leq i \leq n-1. $$

	This function $f_n$ is absolutely continuous and we have $f '_n(t)=(n-1)(v^n_{i+1}-v^n_i)$ for  $t\in\big(t_i^n,t_{i+1}^n\big)$.
	Using the definition of $f '_n$, we obtain the following  regularity properties of $f_n$.
	\begin{lem}\label{lem:fn'} For $n\geq 1$, the curve $f_n$  satisfies:\begin{enumerate}
			\item $\L(f_n)\le  L.$
			\item For all $t,t'\in[0,1]$, $\|f_n(t)-f_n(t')\|\leq L\sqrt{|t-t'|}.$
		\end{enumerate}
	\end{lem}

Asymptotically, the penalty term ensuring that the points $v^n_1,\dots,v^n_n$, $n\geq 1$ belong to the curve $f$ can be neglected.
\begin{lem}\label{lem:pen}
	There exists $c\geq 0$ such that, for all $n\geq 1,$ \begin{equation*}
	\e_n\sum_{i=1}^{n}\|v^n_i-f(t_i^n)\|^2\leq\frac{c}{n}.\end{equation*}	
\end{lem}
	
		\begin{lem}\label{lem:conv-fn}
		The sequence $(f_n)_{n\geq 1}$ converges uniformly to the curve $f$.
	\end{lem}

	Let $\hat{t}_n = t_i^n$ on the event $\{ \hat{X}_n = v^n_i+\xi_i^n\}$, $1 \leq i \leq n$. Note that the sequence $(\hat{t}_{n})_{n\geq 1}$ is bounded.
Thus, up to extending the probability space  $(\Omega, \mathcal F, \P)$, and extracting a subsequence,  we may assume that $(X_n,\hat t_n)$ converge in distribution to a tuple $(X,\hat t)$.
This implies the next result.

\begin{lem}\label{lem:hat-t}There exists a random variable $\hat t$ with values in $[0,1]$, defined on an extension of the probability space  $(\Omega, \mathcal F, \P)$, such that
$$\|X-f(\hat{t})\|=d(X,\mbox{Im}f)\quad a.s.$$	
\end{lem}

In order to be able to state a first order  Euler-Lagrange equation for the criterion \eqref{eq:crit-tot}, we show that the quantity $\E\left[\min_{1\leq i\leq n}\|X_n-x_i-\xi_i^n\|^2\right]$ is differentiable at $x_i$.
Recall the definition \eqref{eq:hatXnx} of $\hat{X}_n^x$. 
 	\begin{lem}\label{lem:diff}The function $(x_1,\dots,x_n)\mapsto \E\left[\min_{1\leq i\leq n}\|X_n-x_i-\xi_i^n\|^2\right]$ is differentiable, and, for $1\leq i\leq n$, the gradient  with respect to $x_i$ is given by
	$$\frac{\partial}{\partial x_i}\E\left[\min_{1\leq j\leq n}\|X_n-x_j-\xi_j^n\|^2\right]=-2\E\left[(X_n-\hat{X}_n^x)\I_{\{\hat{X}_n^x=x_i+\xi_i\}}\right].$$
\end{lem}

The Lagrange multiplier method then leads to the next system of equations satisfied by  $v^n_1,\dots,v^n_n$.

\begin{lem}\label{lem:eqn}For $n\geq 1$, there exists a Lagrange multiplier $\lambda_n\geq 0$ such that
		$$\begin{cases}
	&-\E\left[(X_n-\hat{X}_n)\I_{\{\hat{X}_n=v^n_i+\xi_i^n\}}\right]+\e_n\big(v^n_i-f(t_i^n)\big)+\lambda_n(n-1)(v^n_i-v^n_{i-1}-(v^n_{i+1}-v^n_i))=0,\\&\hfill 2\leq i\leq n-1,\\
	&-\E\left[(X_n-\hat{X}_n)\I_{\{\hat{X}_n=v^n_1+\xi_1^n\}}\right]+\e_n(v^n_1-f(0))-\lambda_n(n-1)(v^n_2-v^n_1)=0,\\
	&-\E\left[(X_n-\hat{X}_n)\I_{\{\hat{X}_n=v^n_n+\xi_n^n\}}\right]+\e_n(v^n_n-f(1))+\lambda_n(n-1)(v^n_n-v^n_{n-1})=0.
	\end{cases}$$
\end{lem}

	\begin{lem}\label{lem:lambda0}
If $\lambda$ is a limit point of the sequence $(\lambda_n)_{n\geq 1}$, then $\lambda\in (0,\infty]$.
	\end{lem}

Hence, up to an extraction, we may assume that the sequence $(\lambda_n)_{n\geq 1}$ converges to a limit $\lambda \in (0,\infty]$.

Let $\delta_\ell$ denote the Dirac mass at $\ell$.	For every $n\geq 2$, we define  $f ''_n$ on $[0,1]$ by \begin{equation}
\label{eq:fn''}
f ''_n=(n-1)\left[\sum_{i=2}^{n-1}(v^n_{i+1}-v^n_i-(v^n_i-v^n_
{i-1}))\delta_{t_i^n}+(v^n_2-v^n_1)\delta_0-(v^n_n-v^n_{n-1})\delta_1\right],
\end{equation}
which is a vector-valued signed measure.

	\begin{lem}
		\label{lem:reg}
		The sequence $(f''_n)_{n\geq 1}$ converges weakly to a signed measure $f''$ on $[0,1]$, with values in $\R^d$, which is the second derivative of $f$.
The following regularity properties hold:
		\begin{itemize}	
			\item $f$ is right-differentiable on $[0,1)$, $\|f '_r(t)\|=L$ for all $t\in[0,1)$,
			\item $f$ is left-differentiable on $(0,1]$, $\|f '_\ell(t)\|=L$ for all $t\in (0,1]$,
			\item $f ''((s,t])=f'_r(t)-f'_r(s)$ for all $0\leq s\leq t<1$,
			\item $f ''([0,1])= 0$. 
	
			\item $f''(\{0\})=f'_r(0)$,
			\item $f''(\{1\})=-f'_\ell(1)$.
		\end{itemize}
	\end{lem}
These properties imply in particular that $\lambda$ is finite. 
\begin{lem}\label{lem:lafini}
We have	$\lambda<\infty$.
\end{lem}
Finally, collecting all the results allows to derive the Euler-Lagrange equation.
\begin{lem}\label{lem:equation}
	For every bounded Borel function $g:[0,1]\to \R^d$, 
	$$
	\E\left[\langle X-f(\hat{t}), g(\hat{t})\rangle\right]=-\lambda\int_{[0,1]}\langle g(t), f ''(dt)\rangle.
	$$ Moreover, $\lambda $ depends only on the curve $f$.
\end{lem}

\begin{proof}[Proof of Lemma \ref{lem:existvi}]
Since $\E[\|X_n\|^2]\leq 2 \E[\|X\|^2]+2 d\zeta_n^2<\infty$ and $\E[\|\xi\|^2]\leq\eta_n^2<\infty$, $F_n$ takes its values in $[0,\infty)$ and is continuous.
	The constraint \eqref{eq:constLvi} defines a nonempty closed set $D_n$.
	Since
	$$ \lim_{ \substack{\|x_1\|+\cdots+\|x_n\| \to\infty\\
	(x_1,\dots,x_n)\in D_n} } F_n(x_1,\dots,x_n) =\infty, $$
	the optimization  problem reduces thus to minimizing a continuous function on a compact set.
	
\end{proof}

\begin{proof}[Proof of Lemma \ref{lem:F-fini}]
	Recall that, for all $t,t'\in[0,1]$, $\|f(t)-f(t')\|\leq L|t-t'|$. Hence, we have \begin{equation*}\label{eq:Lfi}
		(n-1)\sum_{i=2}^{n}\|f(t_i^n)-f(t_{i-1}^n)\|^2\leq L^2,
	\end{equation*} and consequently, we may consider 
	$(x_1,\dots,x_n)=(f(t_1^n),\dots, f(t_n^n))$. We see that \begin{align*}
		F_n(v^n_1,\dots,v^n_n)&\leq \E\left[\|X_n-f(0)-\xi_1^n\|^2\right]\\
		&\leq 2\E\left[\|X_n-\xi_1^n\|^2\right]+2\|f(0)\|^2\\
		&\leq 2\E\left[\|X\|^2\right]+{ 2d\zeta_n^2}+2\eta_n^2+2\|f(0)\|^2.
	\end{align*}
	\end{proof}

\begin{proof}[Proof of Lemma \ref{lem:fn'}]
	 By definition of $f_n'$, and using  that $v^n_1,\dots,v^n_n$ satisfy  constraint \eqref{eq:constLvi}, we have $$\int_0^1\|f '_n(t)\|^2 dt=\sum_{i=1}^{n-1}(n-1)^2\|v^n_{i+1}-v^n_i\|^2\times\frac{1}{n-1}=(n-1)\sum_{i=1}^{n-1}\|v^n_{i+1}-v^n_i\|^2\leq L^2.$$ 
	Hence,	\begin{equation*}\label{eq:Lfn}
	\L(f_n)\le \Bigl(\int_0^1\|f '_n(t)\|^2 dt\Bigr)^{1/2} \le L,
	\end{equation*} and for all $t,t'\in[0,1]$, 
	\begin{equation*} \label{equiuc}
		\|f_n(t)-f_n(t')\|=\Big\|\int_{0}^{1}\I_{[t\land t',t\lor t']}f_n '(u)du\Big\|\leq L\sqrt{|t-t'|}.
	\end{equation*}

\end{proof}

	\begin{proof}[Proof of Lemma \ref{lem:pen}]
The aim is to show that there exists $c\geq 0$ such that, for all $n\geq 1,$ \begin{equation*}
		\e_n\sum_{i=1}^{n}\|v^n_i-f(t_i^n)\|^2\leq\frac{c}{n}.
	\end{equation*} 
	The following upper bound will be useful:
	\begin{equation*}
		\left|\min_{1\leq i\leq n}\left\|X_n-f(t_i^n)-\xi_i^n\right\|-\min_{1\leq i\leq n}\left\|X-f(t_i^n)\right\|\right|\leq \zeta_n\|Z\|+\eta_n.
	\end{equation*}
	By definition of $(v^n_1,\dots,v^n_n)$, since $(f(t_1^n),\dots, f(t_n^n))$ satisfies constraint \ref{eq:constLvi} as already mentioned in the proof of Lemma \ref{lem:F-fini}, we may write
	$$F_n(v^n_1,\dots,v^n_n)\leq \E\left[\min_{1\leq i\leq n}\left\|X_n-f(t_i^n)-\xi_i^n\right\|^2\right].$$

	Observe  that \begin{align*}
		&\left|\min_{1\leq i\leq n}\|X_n-f(t_i^n)-\xi_i^n\|-\min_{t\in[0,1]}\|X-f(t)\|\right|\\&\leq
		\left|\min_{1\leq i\leq n}\|X_n-f(t_i^n)-\xi_i^n\|-\min_{1\leq i\leq n}\|X-f(t_i^n)\|\right|+
		\left|\min_{1\leq i\leq n}\left\|X-f(t_i^n)\right\|-\min_{t\in[0,1]}\|X-f(t)\|\right|
		\\&\leq \zeta_n\|Z\|+\eta_n+ \frac{L}{n-1},
	\end{align*}
	so that \begin{multline*}
		\min_{1\leq i\leq n}\left\|X_n-f(t_i^n)-\xi_i^n\right\|^2\leq \min_{t\in[0,1]}\|X-f(t)\|^2+\left(\eta_n+\zeta_n\|Z\|+\frac{L}{n-1}\right)^2\\+2\left(\eta_n+\zeta_n\|Z\|+\frac{L}{n-1}\right)\min_{t\in[0,1]}\|X-f(t)\|.
	\end{multline*}
	Consequently, there exists $c_1\geq 0$, such that \begin{equation*}F_n(v^n_1,\dots,v^n_n)\leq G(L)+\frac{c_1}{n}.\end{equation*}
	
	Besides, 	$$F_n(v^n_1,\dots,v^n_n)=\E\left[\min_{1\leq i\leq n}\left\|X_n-f_n(t_i^n)-\xi_i^n\right\|^2\right]+\e_n\sum_{i=1}^n\|f_n(t_i^n)-f(t_i^n)\|^2,$$
	and, writing \begin{align*}
		&\left|\min_{1\leq i\leq n}\left\|X_n-f_n(t_i^n)-\xi_i^n\right\|^2- \min_{1\leq i\leq n}\left\|X-f_n(t_i^n)\right\|^2\right|\\&\leq 
		\left|\min_{1\leq i\leq n}\left\|X_n-f_n(t_i^n)-\xi_i^n\right\|-\min_{1\leq i\leq n}\left\|X-f_n(t_i^n)\right\|\right|\times\left(\min_{1\leq i\leq n}\left\|X_n-f_n(t_i^n)-\xi_i^n\right\|+\min_{1\leq i\leq n}\left\|X-f_n(t_i^n)\right\|\right)\\&\leq
		\bigg( \zeta_n\|Z\|+\eta_n\bigg)\bigg(\zeta_n\|Z\|+\eta_n+2\min_{1\leq i\leq n}\left\|X-f_n(t_i^n)\right\|\bigg)\\&=
		\bigg( \zeta_n\|Z\|+\eta_n\bigg)^2+2\bigg( \zeta_n\|Z\|+\eta_n\bigg)\min_{1\leq i\leq n}\left\|X-f_n(t_i^n)\right\|,
	\end{align*}
	we obtain
	\begin{align*}
		F_n(v^n_1,\dots,v^n_n)&\geq 
		\E\left[\min_{1\leq i\leq n}\left\|X-f_n(t_i^n)\right\|^2\right]-\E\left[( \zeta_n\|Z\|+\eta_n)^2\right]-2( \zeta_n\E[\|Z\|]+\eta_n)\E\left[\min_{1\leq i\leq n}\left\|X-f_n(t_i^n)\right\|\right]\\&\quad+\e_n\sum_{i=1}^n\|f_n(t_i^n)-f(t_i^n)\|^2
		\\&\geq 
		\E\left[\min_{t\in[0,1]}\left\|X-f_n(t)\right\|^2\right]- \zeta_n^2\E[\|Z\|^2]-\eta_n^2-2\eta_n\zeta_n\E[\|Z\|]\\&\quad-2( \zeta_n\E[\|Z\|]+\eta_n)\E\left[\min_{1\leq i\leq n}\left\|X-f_n(t_i^n)\right\|\right]+\e_n\sum_{i=1}^n\|f_n(t_i^n)-f(t_i^n)\|^2
		\\&\geq G(L)-\frac{c_2}{n}+\e_n\sum_{i=1}^n\|f_n(t_i^n)-f(t_i^n)\|^2,
	\end{align*} for some constant $c_2\geq 0$.
	Indeed, $\L(f_n)\leq L$ according to point 1 in Lemma \ref{lem:fn'}, which allows to lower bound $\E\left[\min_{t\in[0,1]}\left\|X-f_n(t)\right\|^2\right]$ by $G(L)$, and moreover, 
	$\E\left[\min_{1\leq i\leq n}\left\|X-f_n(t_i^n)\right\|\right]$ is bounded since $(f_n)_{n\geq 1}$ is uniformly bounded and $\E[\|X\|^2]<\infty$.
	Thus, there exists a constant $c_3$ such that  $G(L)-\frac{c_3}{n}+\e_n\sum_{i=1}^n\big\|f_n\big(t_i^n\big)-f\big(t_i^n\big)\big\|^2\leq G(L)+\frac{c_3}{n}$, which shows that $\e_n\sum_{i=1}^n\big\|f_n\big(t_i^n\big)-f(t_i^n)\big\|^2\leq \frac{2c_3}{n}.$
	\end{proof}
	
\begin{proof}[Proof of Lemma \ref{lem:conv-fn}]

	Point 2 in Lemma \ref{lem:fn'} and Lemma \ref{lem:pen} 
	, together with the assumption $n\e_n\to \infty$, imply that the sequence $(f_n)_{n\geq 1}$ converges uniformly to the curve $f$.
	
\end{proof}

	\begin{proof}[Proof of Lemma \ref{lem:hat-t}]

	For every $n\geq 1$, \begin{align*}
		&\Big|\|X_{n}-f_{n}(\hat{t}_{n})\|-\min_{1\leq i\leq n}\big\|X_{n}-f_{n}(t_i^ {n})\big\|\Big|\\&\quad\leq\Big|\|X_{n}-f_{n}(\hat{t}_{n})\|-\min_{1\leq i\leq n} \big\|X_{n}-f_{n}\big(t_i^{n}\big)-\xi_i^{n}\big\|\Big|+\left|\min_{1\leq i\leq n} \big\|X_{n}-f_{n}\big(t_i^{n}\big)-\xi_i^{n}\big\|-\min_{1\leq i\leq n} \big\|X_{n}-f_{n}\big(t_i^{n}\big)\big\|\right|\quad 
		\\&\quad\leq\Big|\|X_{n}-f_{n}(\hat{t}_{n})\|-\sum_{i=1}^n\|X_{n}-f_{n}(\hat{t}_{n})-\xi_i^{n}\|\I_{\{\hat{X}_n=f_n(t_i^n)+\xi_i^n\}}\Big|+\eta_n\quad 
		\\&\quad\leq\sum_{i=1}^n\Big|\|X_{n}-f_{n}(\hat{t}_{n})\|-\|X_{n}-f_{n}(\hat{t}_{n})-\xi_i^{n}\|\Big|\I_{\{\hat{X}_n=f_n(t_i^n)+\xi_i^n\}}+\eta_n\quad 
		\\&\quad\leq 2\eta_n,
	\end{align*}
	Hence, we obtain $$\|X-f(\hat{t})\|=\min_{t\in[0,1]}\|X-f(t)\|\quad a.s.$$
	
\end{proof}

	\begin{proof}[Proof of Lemma \ref{lem:diff}]
		For $x=(x_1,\dots,x_n)\in(\R^d)^n$ and $\omega\in \Omega$, we set 
		$$G_n(x,\omega):=\min_{1\leq i\leq n}\|X_n(\omega)-x_i-\xi_i^n(\omega)\|^2.$$
		For every $x$, since the distribution of $X_n$ gives zero measure to  affine hyperplanes of $\R^ d$ and
		the vectors $x_i+\xi_i^n $, $1\leq i\leq n$, are mutually distinct $\P(d\omega)$ almost surely, we have $\P(d\omega)$ almost surely, $$G_n(x,\omega)
		=\sum_{i=1}^{n}\|X_n(\omega)-x_i-\xi_i^n(\omega)\|^2\I_{\{
			\| {X}_n(\omega)-x_i-\xi_i^n(\omega)\| < \min_{j\not=i} \|{X}_n(\omega)-x_j-\xi_j^n(\omega)\|
			\}}.$$
		
		For every $x\in(\R^ d)^ n$, $\P(d\omega)$ almost surely, $y\mapsto G_n(y,\omega)$ is differentiable at $x$ and for $1\leq i\leq n$,
		\begin{align*}\frac{\partial }{\partial x_i}G_n(x,\omega)&=-	2(X_n(\omega)-x_i-\xi_i^n(\omega))\I_{\{
				\| {X}_n(\omega)-x_i-\xi_i^n(\omega)\| < \min_{j\not=i} \|{X}_n(\omega)-x_j-\xi_j^n(\omega)\|\}}.
			\\&=
			-2(X_n(\omega)-\hat{X}_n^x(\omega))\I_{\{\hat{X}_n^x(\omega)=x_i+\xi_i^n(\omega)\}}.
		\end{align*}
		For every $u=(u_1,\dots,u_n)\in(\R^d)^n$, we set $\|u\|=(\sum_{i=1}^{n}\|u_i\|^2)^{1/2}.$
		 Let $x^{(k)}=(x_1^{(k)},\dots,x_n^{(k)})$ be a sequence tending to $x=(x_1,\dots,x_n)\in(\R^d)^n$ as $k$ tends to infinity. Then, 	
		$$\left[G_n(x^{(k)},\omega)-G_n(x,\omega)-\sum_{i=1}^{n}\left\langle \frac{\partial }{\partial x_i}G_n(x,\omega),x_i^{(k)}-x_i\right\rangle\right]\times \frac 1{\|x^{(k)}-x\|}$$ converges  $\P(d\omega)$ almost surely to 0 as $k$ tends to infinity.
		Moreover, 
		\begin{align*}
			&\left| G_n(x,\cdot)-G_n(x^{(k)},\cdot)  \right|
			\\&\quad= \Big(\min_{1\leq i\leq n}\|X_n-x_i-\xi_i^n\|+\min_{1\leq i\leq n}\|X_n-x^{(k)}_i-\xi_i^n\|\Big)\Big|\min_{1\leq i\leq n}\|X_n-x_i-\xi_i^n\|-\min_{1\leq i\leq n}\|X_n-x^{(k)}_i-\xi_i^n\|\Big|\\&\quad\leq 2\left(\|X_n\|+\eta_n+\|x_1\|+\|x_1^{(k)}\|\right)\max_{1\leq i\leq n}\|x_i-x^{(k)}_i\|,
		\end{align*}
		so that $$ \frac{\left| G_n(x,\cdot)-G_n(x^{(k)},\cdot)  \right|}{\|x-x^{(k)}\|}\leq C(\|X_n\|+1),$$ where $C$ is a constant which does not depend on $k$. Similarly, we have, for $1\leq i\leq n$, $$\left\|\frac{\partial}{\partial x_i}G_n(x,\cdot)\right\|\leq C'(\|X_n\|+1),$$
		where $C'$ does not depend on $k$, and, thus, \begin{align*}
			\frac{1}{\|x-x^{(k)}\|}\left|\sum_{i=1}^{n}\left\langle \frac{\partial }{\partial x_i}G_n(x,),x_i^{(k)}-x_i\right\rangle\right|&\leq C'(\|X_n\|+1)\frac{\sum_{i=1}^{n}\|x_i^{(k)}-x_i\|}{\|x-x^{(k)}\|}\\&\leq C'\sqrt{n}(\|X_n\|+1).
		\end{align*}
		Since  $\E[\|X_n\|]<\infty$, the result follows from  Lebesgue's dominated convergence theorem.
	\end{proof}

	\begin{proof}[Proof of Lemma \ref{lem:eqn}]
	By Lemma \ref{lem:diff}, we obtain that $F_n$ is differentiable, and for $1\leq i\leq n$, the gradient  with respect to $x_i$ is given by 
	$$\frac{\partial}{\partial x_i}F_n(x_1,\dots,x_n)=-2\E\left[(X_n-\hat{X}_n^x)\I_{\{\hat{X}_n^x=x_i+\xi_i^n\}}\right]+2\e_n\big(x_i-f(t_i^n)\big),\quad 1\leq i\leq n.$$ 
	
	Consequently, considering 
	the  minimization of $F_n$ under the length constraint \eqref{eq:constLvi}, there exists a Lagrange multiplier $\lambda_n\geq0$ such that 
	$$\begin{cases}
	&-2\E\left[(X_n-\hat{X}_n)\I_{\{\hat{X}_n=v^n_i+\xi_i^n\}}\right]+2\e_n\big(v^n_i-f(t_i^n)\big)+2\lambda_n(n-1)(v^n_i-v^n_{i-1}-(v^n_{i+1}-v^n_i))=0,\\&\hfill 2\leq i\leq n-1,\\
	&-2\E\left[(X_n-\hat{X}_n)\I_{\{\hat{X}_n=v^n_1+\xi_1^n\}}\right]+2\e_n(v^n_1-f(0))-2\lambda_n(n-1)(v^n_2-v^n_1)=0,\\
	&-2\E\left[(X_n-\hat{X}_n)\I_{\{\hat{X}_n=v^n_n+\xi_n^n\}}\right]+2\e_n(v^n_n-f(1))+2\lambda_n(n-1)(v^n_n-v^n_{n-1})=0,
	\end{cases}$$
	that is,
	$$\begin{cases}
	&-\E\left[(X_n-\hat{X}_n)\I_{\{\hat{X}_n=v^n_i+\xi_i^n\}}\right]+\e_n\big(v^n_i-f(t_i^n)\big)+\lambda_n(n-1)(v^n_i-v^n_{i-1}-(v^n_{i+1}-v^n_i))=0,\\&\hfill 2\leq i\leq n-1,\\
	&-\E\left[(X_n-\hat{X}_n)\I_{\{\hat{X}_n=v^n_1+\xi_1^n\}}\right]+\e_n(v^n_1-f(0))-\lambda_n(n-1)(v^n_2-v^n_1)=0,\\
	&-\E\left[(X_n-\hat{X}_n)\I_{\{\hat{X}_n=v^n_n+\xi_n^n\}}\right]+\e_n(v^n_n-f(1))+\lambda_n(n-1)(v^n_n-v^n_{n-1})=0.
	\end{cases}$$
	\end{proof}

	\begin{proof}[Proof of Lemma \ref{lem:lambda0}]

	Let $g:[0,1]\to\R^d$ be an absolutely continuous function such that $\int_{0}^{1}\|g'(t)\|^2 dt<\infty$. For $n\geq 1$, we may write 
	\begin{align}
		\E&[\langle X_n-f_n(\hat{t}_n), g(\hat{t}_n)\rangle]\nonumber\\&=\sum_{i=1}^{n}\left\langle\E\left[(X_n-\hat{X}_n+
		\xi_i^n)\I_{\{\hat{X}_n=v^n_i+
			\xi_i^n\}}\right], g(t_i^n)\right\rangle\nonumber \\&=
		\sum_{i=1}^{n}\left\langle \E\left[(X_n-\hat{X}_n
		)\I_{\{\hat{X}_n=v^n_i+
			\xi_i^n\}}\right], g(t_i^n)\right\rangle+\sum_{i=1}^{n}\left\langle\E\left[
		\xi_i^n\I_{\{\hat{X}_n=v^n_i+
			\xi_i^n\}}\right], g(t_i^n)\right\rangle\nonumber\\
		&=\sum_{i=1}^{n}\llangle \E\left[
		\xi_i^n\I_{\{\hat{X}_n=v^n_i+
			\xi_i^n\}}\right], g(t_i^n)\rrangle+\e_n\sum_{i=1}^{n}\langle v^n_i-f(t_i^n), g(t_i^n)\rangle\nonumber\\&\quad+\lambda_n(n-1)\left[-\langle v^n_2-v^n_1, g(0)\rangle+\sum_{i=2}^{n-1}\langle v^n_{i}-v^n_{i-1}-(v^n_{i+1}-v^n_i), g(t_i^n)\rangle+\langle v^n_n-v^n_{n-1}, g(1)\rangle\right]\nonumber\\&=\sum_{i=1}^{n}\llangle \E\left[
		\xi_i^n\I_{\{\hat{X}_n=v^n_i+
			\xi_i^n\}}\right], g(t_i^n)\rrangle+\e_n\sum_{i=1}^{n}\langle v^n_i-f(t_i^n), g(t_i^n)\rangle\nonumber\\&\quad+\lambda_n(n-1)\Bigg[\sum_{i=1}^{n-2}\langle v^n_{i+1}-v^n_{i}, g(t_{i+1}^n)\rangle-\sum_{i=2}^{n-1}\langle v^n_{i+1}-v^n_i, g(t_i^n)\rangle -\langle v^n_2-v^n_1, g(0)\rangle\nonumber\\&\quad+\langle v^n_n-v^n_{n-1}, g(1)\rangle\Bigg]\nonumber\\&=\sum_{i=1}^{n}\llangle \E\left[
		\xi_i^n\I_{\{\hat{X}_n=v^n_i+
			\xi_i^n\}}\right], g(t_i^n)\rrangle\nonumber\\&\quad+\e_n\sum_{i=1}^{n}\langle v^n_i-f(t_i^n), g(t_i^n)\rangle+\lambda_n(n-1)\sum_{i=1}^{n-1}\langle v^n_{i+1}-v^n_i,g(t_{i+1}^n)-g(t_i^n)\rangle.\label{eq:suml}
	\end{align}
	Note first that
	\begin{align}
		\left|\sum_{i=1}^{n}\llangle \E\left[
		\xi_i^n\I_{\{\hat{X}_n=v^n_i+
			\xi_i^n\}}\right], g(t_i^n)\rrangle\right|&\leq\eta_n\|g\|_\infty
		\sum_{i=1}^{n}\E\left[
		\I_{\{\hat{X}_n=v^n_i+
			\xi_i^n\}}\right]\nonumber\\&=\eta_n\|g\|_\infty.\label{eq:Oetan}
	\end{align}
	Then,
	\begin{align}
		\left| \e_n\sum_{i=1}^{n}\langle v^n_i-f(t_i^n), g(t_i^n)\rangle\right|&\leq \e_n \sum_{i=1}^{n} \|v^n_i-f(t_i^n)\|\|g\|_\infty\nonumber\\
		&\leq \e_n \Bigl(\sum_{i=1}^{n}\|v^n_i-f(t_i^n)\|^2 \bigr)^{1/2}\sqrt{n}\|g\|_\infty\nonumber\\&\leq \sqrt{c\e_n}\|g\|_\infty,\label{eq:Oepsn}
	\end{align}according to Lemma \ref{lem:pen}. Regarding the last term, we may write 
	\begin{align*}
		\left|(n-1)\sum_{i=1}^{n-1}\langle v^n_{i+1}-v^n_i,g(t_{i+1}^n)-g(t_i^n)\rangle\right|&\leq (n-1) \left[\sum_{i=1}^{n-1}\left\|v^n_{i+1}-v^n_i\right\|^2\sum_{i=1}^{n-1}\left\|g(t_{i+1}^n)-g(t_i^n)\right\|^2\right]^{1/2}\\&\leq
		L\sqrt{n-1}\left[\sum_{i=1}^{n-1}\Big\|\int_{t_i^n}^{t_{i+1}^n}g'(t)dt\Big\|^2\right]^{1/2}\\&\leq L \left[\int_{0}^{1}\|g'(t)\|^2dt\right]^{1/2}.\end{align*}
	
	Thus, if $h:\R^d\to\R^d$ is continuously differentiable, we have \begin{align*}
		|\E[\langle X_n-f_n(\hat{t}_n), h(f_n(\hat t_n))\rangle]|&\leq \sqrt{c\e_n}\|h\|_\infty+\lambda_n L \left[\int_{0}^{1}\|\nabla h(f_n(t)), f_n '(t)\|^2dt\right]^{1/2}\\&\leq\sqrt{c\e_n}\|h\|_\infty+\lambda_nL
		\sup_{t\in[0,1]}\|\nabla h(f_n(t))\|
		\left[\int_{0}^{1}\|f_n'(t)\|^2dt\right]^{1/2}\\&\leq
		\sqrt{c\e_n}\|h\|_\infty+\lambda_nL^2
		\sup_{t\in[0,1]}\|\nabla h(f_n(t))\|.
	\end{align*}Recall that $(X_n,\hat{t}_n)$ is assumed to converge to $(X,\hat{t})$. Since $\e_n\to 0$ and $(f_n)_{n\geq 1}$ is uniformly bounded, we see that $\lambda=0$ would imply that  $$\E[\langle X-f(\hat{t}), h(f(\hat{t}))\rangle]
=0,$$ so that $\E[X-f(\hat{t})|f(\hat{t})]=0$ a.s. by density of continuously differentiable functions since $h$ is an arbitrary such function. This contradicts Lemma \ref{lem:noneq}.\end{proof}

	\begin{proof}[Proof of Lemma \ref{lem:reg}]
		
For an $\R^d$-valued signed measure $m=(m^1, \dots,m^d)$ on $[0,1]$, we set
\begin{equation} \label{normnu}
\| m \| = \Bigl( \sum_{j=1}^d \| m^j \|_{TV}^2 \Bigr)^{1/2}
\end{equation}
where $\| m^j \|_{TV}$ denotes the total variation norm of $m^j$.
Recall that $$
f ''_n=(n-1)\left[\sum_{i=2}^{n-1}(v^n_{i+1}-v^n_i-(v^n_i-v^n_
{i-1}))\delta_{t_i^n}+(v^n_2-v^n_1)\delta_0-(v^n_n-v^n_{n-1})\delta_1\right].
$$
Thanks to the Euler-Lagrange system of equations obtained in Lemma \ref{lem:eqn},	we may write
	\begin{align*}
		\lambda_n\times\|f''_n\|&\leq\lambda_n\sum_{i=1}^n\|f''_n(\{t^n_i\})\|\\&\leq 
		\sum_{i=1}^{n}\left\|\E\left[(X_n-\hat{X}_n)\I_{\{\hat{X}_n=v^n_i+\xi_i^n\}}\right]\right\|+\e_n \sum_{i=1}^{n}\|v_i ^n -f(t_i^n)\|\\&\leq \E[\|X_n-\hat{X}_n\|]+\e_n\sqrt{n}\left(\sum_{i=1}^{n}\|v^n_i-f(t_i^n)\|^2\right)^{1/2}\\
		&\le F_n(v^n_1,\dots,v^n_n)^{1/2} +\e_n\sqrt{n}\left(\sum_{i=1}^{n}\|v^n_i-f(t_i^n)\|^2\right)^{1/2}.
	\end{align*}
	Consequently,  using Lemma \ref{lem:F-fini} and  Lemma \ref{lem:pen},  $\e_n\to 0$ and  $\lim_{n\to\infty} \lambda_{n}=\lambda\in(0,\infty]$, we obtain that  $\sup_{n\geq 1}\|f''_n\|<\infty$, that is, the sequence of signed measures  $(f''_n)_{n\geq 1}$ is uniformly bounded in total variation norm.
	Hence, it is relatively compact for the topology induced by the bounded Lipschitz norm defined for every signed measure $m$ by $$\|m\|_{\mathrm{BL}}=\sup\left\{\Big\|\int g(x)m(dx)\Big\|, \|g\|_\infty\leq 1, \sup_{t\neq x}\frac{|g(x)-g(t)|}{|x-t|}\leq 1\right\} .$$

	Let us show that the sequence $(f''_{n})_{n\geq 1}$ converges weakly to some signed measure. Let $\nu$ be a limit point of $(f''_{n})_{n\geq 1}$.
There exists an increasing function $\sigma:\N\to\N$, such that, for every $(s,t)$ such that $\nu(\{s\})=\nu(\{t\})=0$, 
	\begin{align}
	&f''_{\sigma(n)}((s,t])\to \nu((s,t]),\label{eq:cvnu1}\\
		&f''_{\sigma(n)}([0,t])\to \nu([0,t]),\quad f''_{\sigma(n)}([0,t))\to \nu([0,t)).\label{eq:cvnu2}
	\end{align}
	Since, for $0\leq s \leq t<1$, $f''_n((s,t])=f '_{n,r}(t)-f '_{n,r}(s)$, we have, for  $0\leq t<1$, $$f_n(t)=f_n(0)+tf '_{n,r}(0)+ \int_{0}^{t}f ''_n((0,u])du.$$
	Note that $f'_{n,r}(0)=f''_n(\{0\})$, so that the fact that $\sup_{n\geq 1}\|f''_n\|<\infty$ implies in particular that $(f'_{n,r}(0))_{n\geq 1}$ is bounded.
	Thus, up to an extraction, by \eqref{eq:cvnu1}, all terms converge: there exists a vector $v\in\R^d$, such that, for $0 \leq t<1$, $$f(t)=f(0)+tv+ \int_{0}^{t}\nu((0,u])du.$$ 
	Consequently, $v=f'_{r}(0)$, and, for $0\leq s \leq t<1$, $$\nu((s,t])=f '_{r}(t)-f '_{r}(s).$$ In other words, the signed measure $\nu$ is the second derivative of $f$ in the distribution sense, called hereafter $f''$, and  $(f''_{n})_{n\geq 1}$ converges weakly to $f''$.

	Observe, on the definition \eqref{eq:fn''}, that $f''_n([0,1])=0$, so that $f''([0,1])=0$.
	
	We have, for $t$ such that $f''(\{t\})=0$, $$f''_n([0,t])\to f''([0,t]),\quad f''_n([0,t))\to f''([0,t)).$$
Hence, since, for $t\in[0,1)$, $f''_n([0,t])=f '_{n,r}(t)$, and $t\mapsto f ''([0,t])$ is right-continuous, $f_r '(t)=f ''([0,t])$ for $t\in [0,1)$. Similarly, for $t\in (0,1]$, $f''_n([0,t))=f '_{n,\ell}(t)$, and $t\mapsto f ''([0,t))$  is left-continuous, so that $f_\ell '(t)=f ''([0,t))$ for $t\in (0,1]$.

	Recall that $f$ is  $L$-Lipschitz.
	Moreover, according to Corollary \ref{cor:L(f)},  $\L(f)= L$ since $G(L)>0$. Thus, 
	we have $\|f '_r(t)\|=L$ $dt-$a.e., and, since $f '_r$ is right-continuous, this implies that $\|f '_r(t)\|=L$ for all $t\in [0,1)$. Similarly, we obtain that  $\|f '_\ell(t)\|=L$ for all $t\in (0,1]$.
\end{proof}

\begin{proof}[Proof of Lemma \ref{lem:lafini}]
Observe that $f''\neq 0$. Indeed, we have, for example, $f''(\{0\})=f'_r(0)$, with $\|f'_r(0)\|=L>0 $. Yet,   $\lambda=\infty$ would imply 	$f''=0$ since $\sup_{n\geq 1}\left(\lambda_n\times \|f''_n\|\right)<\infty$.
\end{proof}

\begin{proof}[Proof of Lemma \ref{lem:equation}]

	Clearly, it suffices to consider the case where the test function $g$ is continuous.
	Using  equation \eqref{eq:suml} and the upper bounds \eqref{eq:Oetan} and \eqref{eq:Oepsn} in the proof of Lemma \ref{lem:lambda0}, we obtain, for $n\geq 2$,
	$$\left|\E[\langle X_n-f_n(\hat{t}_n), g(\hat{t}_n)\rangle]-\lambda_n(n-1)\sum_{i=1}^{n-1}\langle v^n_{i+1}-v^n_i,g(t_{i+1}^n)-g(t_i^n)\rangle \right|\leq (\eta_n+c\sqrt{\e_n})\|g\|_\infty,$$
	and besides
	$$ \lambda_n(n-1)\sum_{i=1}^{n-1}\langle v^n_{i+1}-v^n_i,g(t_{i+1}^n)-g(t_i^n)\rangle=-\lambda_n\int_{[0,1]}\langle g(t), f ''_n(dt) \rangle.$$
	 Thus,  passing to the limit, we see that $f$ satisfies  equation \eqref{eq:formuleth}.
	 
	 Finally, the uniqueness of $\lambda$ follows from the uniqueness of $\hat X$ (Proposition \ref{prop:hatXneg}), and the fact that
	\begin{equation*}
	\E[\langle X-\hat{X}, \hat{X}\rangle]=\lambda\int_{0}^{1}\|f_r'(s)\|^2ds= \lambda L^2
	\end{equation*}
	obtained thanks to equation \eqref{eq:formuleIP} in Remark \ref{rem:formuleIP}.
	
		\end{proof}

\section{An application: injectivity of $f$}\label{section:inj}

In this section, we present an application of the formula \eqref{eq:formuleth} of Theorem \ref{theo:main}.
We will use this first order condition to show in dimension $d=2$ that an open optimal curve is injective, and a closed optimal curve restricted to $[0,1)$ is injective, except in the case where its image is a segment. To obtain the result, we  follow arguments exposed in \cite{LuSle} in the frame of the penalized problem, for open curves. The main difference is the fact that we have at hand the Euler-Lagrange equation, which allows to simplify the proof. 

Again, we consider $L>0$ such that $G(L)>0$ and a curve  
$f\in\mathcal C_L$ such that $\Delta(f)=G(L)$, which is $L$-Lipschitz. We let $\hat{t}$ be defined as in  Theorem \ref{theo:main}. The random  vector $f(\hat{t})$ will sometimes be denoted by $\hat{X}$. Recall that $\|X-\hat{X}\|=d(X, \mbox{Im}f)$ a.s. by Theorem \ref{theo:main}.

To prove the injectivity of $f$, we will need several preliminary lemmas. Let us point out that Lemma \ref{lem:turn} to Lemma \ref{lem:tangent} below are valid for every $d\geq 1$.

First of all, we state the next  lemma, which will be useful in the sequel, providing a lower bound on the  curvature of any closed arc of $f$. Recall that the  total variation of a  signed measure $\nu$ is defined by
\begin{equation*}
\| \nu \| = \Bigl( \sum_{j=1}^d \| \nu^j \|_{TV}^2 \Bigr)^{1/2},
\end{equation*}
where $\| \nu^j \|_{TV}$ denotes the total variation norm of $\nu^j$. For a Borel set $A\subset [0,1]$, $f''_A$ denotes the vector-valued signed measure defined by $f''_A(B)=f''(A\cap B)$ for all Borel set $B\subset[0,1]$.

\begin{lem}\label{lem:turn}
	If $0\leq a<b\leq 1$ and $f(a)=f(b)$, then $\|f''_{(a,b]}\| \geq L.$
\end{lem}

\begin{proof}[Proof of Lemma \ref{lem:turn}]
	Let us write \begin{multline*}
	0=f(b)-f(a)=\int_{a}^{b}f'_r(t)dt\\=\int_{a}^{b}\left[f_r'(0)+\int_{(0,t]} f''(ds)\right]dt=(b-a)f_r'(0)+\int_{(0,b]}(b-(s\lor a)) f''(ds)\\=(b-a)f_r'(0)+(b-a)f''((0,a])+\int_{(a,b]}(b-s)f''(ds)
	\\=(b-a)f'_r(a)+\int_{(a,b]}(b-s)f''(ds).
	\end{multline*}
	Thus, $\int_{(a,b]}\frac{b-s}{b-a}f''(ds)=-f'_r(a),$ which implies $\|f''_{(a,b]}\| \geq \|f'_r(a)\|=L.$
\end{proof}

 As a first step toward injectivity, we now show that, if a point is multiple, it is only visited finitely many times.  

\begin{lem}\label{lem:doublefini}For every $t\in [0,1],$ the set $f^{-1}(\{f(t)\})$ is finite.
	\end{lem}
	\begin{proof}
		Let $t\in[0,1]$. Suppose that $f^{-1}(\{f(t)\})$ is infinite. Then, for all $k\geq 1$, there exist $t_0,t_1,\dots, t_k\in f^{-1}(\{f(t)\})$  such that $0\leq t_0<t_1<\dots <t_k\leq 1$. So, by Lemma \ref{lem:turn}, $\|f''\|\geq \sum_{i=1}^{k}\|f''_{(t_{i-1},t_i]}\|\geq kL$, which contradicts the fact that $f$ has finite curvature.
	\end{proof}

In the case $\mathcal C_L=\{\phi:[0,1]\to \R^d,\L(\phi)\leq L\}$, the endpoints of the curve $f$ cannot be multiple points.

\begin{lem}\label{lem:extpasdb}Let $\mathcal C_L=\{\phi:[0,1]\to \R^d,\L(\phi)\leq L\}$. We have $f^{-1}(\{f(0)\})= \{0\}$ and  $f^{-1}(\{f(1)\})= \{1\}$.
	\end{lem}
	
	\begin{proof}Observe that, by symmetry, we only need to prove the first statement since  the second one follows then by considering the curve $t\mapsto f(1-t)$.
		Assume that the set $f^{-1}(\{f(0)\})$ has cardinality at least 2. Thanks to Lemma \ref{lem:doublefini}, we may consider $t_0=\min\{t>0: f(t)=f(0)\}$. For $x\in \mbox{Im} f,$ we set $ \hat{t}(x)=\inf\{t\in[0,1], f(t)=x\}$. For every $\e\in(0,t_0)$, we let $$\hat{X}_\e=f\big(\hat{t}\lor\e\big)\I_{\{\hat{t}>0\}}+f(0)\I_{\{\hat{t}=0\}} .$$
		With this definition, the random  vector $\hat{X}_\e$ takes its values in $f([\e,1])\cup \{f(0)\}$, that is in $f([\e,1])$ since $f(t_0)=f(0)$ and $\e<t_0$. Thus, $\frac{\hat{X}_\e}{1-\e}$ takes its values in $\frac{f([\e,1])}{1-\e}$, which is the image of  a curve with length at most $L$. Consequently, by optimality of $f$, we have $$\E\left[\bigg\|X-\frac{\hat{X}_\e}{1-\e}\bigg\|^2\right]\geq \E[\|X-\hat{X}\|^2].$$
		Besides, we may write 
		\begin{align*}
		\bigg\|X-\frac{\hat{X}_\e}{1-\e}\bigg\|^2&=\bigg\|X-\hat{X}+\hat{X}-\frac{\hat{X}_\e}{1-\e}\bigg\|^2\\&=\|X-\hat{X}\|^2+\bigg\|\hat{X}-\frac{\hat{X}_\e}{1-\e}\bigg\|^2+2\bigg\langle X-\hat{X}, \hat{X}-\frac{\hat{X}_\e}{1-\e}\bigg\rangle
		\\&=\|X-\hat{X}\|^2+\frac{1}{(1-\e)^2}\|\hat{X}-\hat{X}_\e-\e\hat{X}\|^2+\frac{2}{1-\e}\left(\langle X-\hat{X}, \hat{X}-\hat{X}_\e\rangle-\e\langle X-\hat{X},\hat{X}\rangle\right).
		\end{align*}
		As $\|\hat{X}-\hat{X}_\e\|\leq L\e$ since $f$ is $L$-Lipschitz, we get 
		\begin{align*}
		\E[\|\hat{X}-\hat{X}_\e-\e\hat{X}\|^2]&\leq 2L^2\e^2+2\e^2\E[\|\hat{X}\|^2]=2(L^2+ \E[\|\hat{X}\|^2])\e^2.
		\end{align*}
		 Note that $\E[\|\hat{X}\|^2]<\infty$
	by the same argument as in $\eqref{eq:hatXfini}$.
	Moreover,  thanks to equation \eqref{eq:formuleIP} in Remark \ref{rem:formuleIP}, we have
	\begin{equation}\label{eq:formulehatX}
	\E[\langle X-\hat{X}, \hat{X}\rangle]=\lambda\int_{0}^{1}\|f_r'(s)\|^2ds= \lambda L^2.
	\end{equation}
	Furthermore, $\hat{X}-\hat{X}_\e=(f(\hat{t})-f(\e))\I_{\{0<\hat{t}\leq \e\}}$, so that equation \eqref{eq:formuleth} implies
	$$\E[\langle X-\hat{X},\hat{X}-\hat{X}_\e\rangle]=-\lambda \int_{[0,1]} \langle (f(t)-f(\e))\I_{\{0<t\leq \e\}} ,f''(dt)\rangle.$$
	Hence, \begin{align*}
	|\E[\langle X-\hat{X},\hat{X}-\hat{X}_\e\rangle]|&
	\le \lambda\sum_{j=1}^d \int_{(0,\e]} | f^j(t)-f^j(\e)|\, |(f'')^j|(dt)\\
	&\le \lambda L\e \sum_{j=1}^d |(f'')^j| ((0,\e]),
		\end{align*}
		where $| (f'')^j|$ stands for the total variation of the signed measure $(f'')^j$. 
		Finally, we obtain 
		$$ \E\left[\bigg\|X-\frac{\hat{X}_\e}{1-\e}\bigg\|^2\right]\leq \E\left[\|X-\hat{X}\|^2\right]+2(L^2+\E[\|\hat{X}\|^2])\e^2+\lambda L\e\rho(\e) -\frac{2\e}{1-\e}\lambda L^2,$$ where $\rho(\e)$ tends to 0 as $\e\to 0$. This inequality shows that, for $\e$ small enough, $\E\left[\big\|X-\frac{\hat{X}_\e}{1-\e}\big\|^2\right]<\E[\|X-\hat{X}\|^2]$,  which contradicts the optimality of $f$.
			\end{proof}

		For an open curve,
		there exists a multiple point which is the last multiple point.
		\begin{lem}\label{lem:lastdouble}
			Let $\mathcal C_L=\{\phi:[0,1]\to \R^d,\L(\phi)\leq L\}$. There exists $\delta>0$ such that for every $t\in[1-\delta,1],$ $f^{-1}(\{f(t)\})=\{t\}.$
		\end{lem}
		\begin{proof}
Otherwise, we can build sequences $(t_k)_{k\geq 1}$ and $(s_k)_{k\geq 1}$ such that $t_k\to 1$ and $f(t_k)=f(s_k)$, with $s_k\neq t_k$ for all $k\geq 1.$ Up to extraction of a subsequence, we may assume that	$(s_k)$	 converges to a limit $s\in[0,1]$. Hence, we have $f(s)=f(1)$, which implies $s=1$ by Lemma \ref{lem:extpasdb}. Up to another extraction, we may consider that the intervals
$[s_k\land t_k,s_k\lor t_k]$, $k\geq1$, are mutually disjoint. 
Finally, using Lemma \ref{lem:turn}, we obtain $$\|f''\|
\geq\sum_{k\geq 1}\|f''_{(s_k\land t_k,s_k\lor t_k]}\|=\infty,$$ which yields a contradiction since we have shown that an optimal curve has finite curvature.

\end{proof}

Now, we show that the two branches of the curve are necessarily tangent at a multiple point.
\begin{lem}\label{lem:tangent}
\begin{enumerate}
	\item [$(i)$] If there exist $0<t_0<t_1<1$ such that $f(t_0)=f(t_1)$, then $f'_\ell(t_0)=f'_r(t_0)=- f'_r(t_1)=-f'_\ell(t_1)$.
\item [$(ii)$] In the case $\mathcal C_L=\{\phi:[0,1]\to \R^d,\L(\phi)\leq L, \phi(0)=\phi(1)\}$, if there exists $0<t<1$ such that $f(t)=f(0)$, then $f'_\ell(t)=f'_r(t)=- f'_r(0)=-f'_\ell(1)$.
\end{enumerate}
\end{lem}

\begin{proof}
First, we show that point $(ii)$ follows from point $(i)$. Let $t\in(0,1)$ such that $f(t)=f(0)$.
Define the curve $g$ by $g(s)=f(s+t/2)$ for $s\in[0,1-t/2]$ and $g(s)=f(s+t/2-1)$ for $s\in[1-t/2,1]$.
Clearly, $g$ is a closed curve, $\Delta(g)=\Delta(f)$ and $g$ is $L$-Lipschitz. Moreover, one has: $g(t/2)=g(1-t/2)$, $g'_r(t/2)=f'_r(t)$, $g'_\ell(t/2)=f'_\ell(t)$, $g'_r(1-t/2)=f'_r(0)$ and $g'_\ell(1-t/2)=f'_\ell(1)$.
Consequently, if $(i)$ holds true for $g$, one deduces $(ii)$.

It remains to show point $(i)$. Suppose that $f'_\ell(t_0)\neq f'_r(t_0)$.
Let $\gamma\in(0,1]$ and $\e>0$. We introduce the random vectors $\hat{X}_{0,\gamma}=(1+\gamma)\hat{X}$ and $$\hat{X}_{\e,\gamma}=(1+\gamma)\left[\hat{X}\I_{\hat{t}\in[0,t_0-\e)\cup (t_0+\e, 1]\cup\{t_0\}}+h_\e(\hat{t})\I_{\hat{t}\in[t_0-\e,t_0+\e]\setminus \{t_0\}}\right],$$ 
where $h_\e(t)=\left(\frac{f(t_0+\e)-f(t_0-\e)}{2\e}(t-(t_0-\e))+f(t_0-\e)\right)$. 

Let us write \begin{align}
\E[\|X-\hat{X}_{0,\gamma}\|^2]&=\E[\|X-\hat{X}\|^2]+\E[\|\hat{X}-\hat{X}_{0,\gamma}\|^2]+2\E[\langle X-\hat{X}, \hat{X}-\hat{X}_{0,\gamma}\rangle]\nonumber\\&=
\E[\|X-\hat{X}\|^2]+\gamma^2\E[\|\hat{X}\|^2]-2\E[\langle X-\hat{X}, \gamma\hat{X}\rangle]
\nonumber\\&=
\E[\|X-\hat{X}\|^2]+\gamma^2\E[\|\hat{X}\|^2]-2\gamma\lambda L^2.\label{eq:XhatXg0}
\end{align}For the last equality, we used equation \eqref{eq:formulehatX}. \\
Note that $\hat{X}_{\e,\gamma}=\hat{X}_{0,\gamma}+(1+\gamma)(h_\e(\hat{t})-f(\hat{t}))\I_{\hat{t}\in[t_0-\e,t_0+\e]\setminus \{t_0\}}$ and that $\|h_\e(\hat{t})-f(\hat{t})\|\leq 4\e L$. So, we have
\begin{multline}
\E[\|X-\hat{X}_{\e,\gamma}\|^2]=\E[\|X-\hat{X}_{0,\gamma}\|^2]+(1+\gamma)^2\E[\|h_\e(\hat{t})-f(\hat{t})\|^2\I_{\hat{t}\in[t_0-\e,t_0+\e]\setminus \{t_0\}}]\\+2(1+\gamma) \E [\langle X-\hat{X}_{0,\gamma},(h_\e(\hat{t})-f(\hat{t}))\I_{\hat{t}\in[t_0-\e,t_0+\e]\setminus \{t_0\}}\rangle]\\ =\E[\|X-\hat{X}_{0,\gamma}\|^2]+\mathcal O(\e^2)+  o(\e).\label{eq:XhatXge}
\end{multline}
Indeed, $\P([t_0-\e,t_0+\e]\setminus \{t_0\})$ tends to 0 as $\e$ tends to 0.
Besides, the random vector $\hat{X}_{\e,\gamma}$ is taking its values in the image of a curve of length $$L_{\e,\gamma}:=(1+\gamma)(L(1-2\e)+\|f(t_0+\e)-f(t_0-\e)\|).$$
Yet, since $f'_\ell(t_0)\neq f'_r(t_0)$, if $\e$ is small enough, there exists $\alpha\in[0,1)$ such that \begin{align*}
\|f(t_0+\e)-f(t_0-\e)\|^2&=\|f(t_0+\e)-f(t_0)+f(t_0)-f(t_0-\e)\|^2\\
&=\e^2\bigg[\Big\|\frac{f(t_0+\e)-f(t_0)}{\e}\Big\|^2+\Big\|\frac{f(t_0)-f(t_0-\e)}{\e}\Big\|^2\bigg.\\&\bigg.\quad+2\Big\langle \frac{f(t_0+\e)-f(t_0)}{\e}, \frac{f(t_0)-f(t_0-\e)}{\e}\Big\rangle\bigg].\\
&\leq \e^2(2L^2+2L^2\alpha).
\end{align*}
Hence, $\|f(t_0+\e)-f(t_0-\e)\|<\e L\sqrt{2(1+\alpha)}$, and, thus,
 $$L_{\e,\gamma}\leq (1+\gamma)(L-2\e L+\e L\sqrt{2(1+\alpha)})=(1+\gamma)(L-\eta\e),$$where $\eta>0.$
Let $\gamma=\frac{\eta\e}{L}$. Then, for $\e$ small enough, we get $L_{\e,\gamma}\leq L-\frac{(\eta\e)^2}{L}<L$ and, using equations \eqref{eq:XhatXg0} and \eqref{eq:XhatXge}, we have  $\E[\|X-\hat{X}_{\e,\gamma}\|^2]<\E[\|X-\hat{X}\|^2].$ This  contradicts the optimality of $f$.
So, $f'_\ell(t_0)= f'_r(t_0)$. Similarly, we obtain that $f'_\ell(t_1)= f'_r(t_1)$.
Finally, consider the curve $g$, defined by $$g(t)=\begin{cases}f(t)&\mbox{if }t\in[0,t_0]\cup [t_1,1]\\f(t_0+t_1-t)&\mbox{if }t\in(t_0,t_1).\end{cases}$$ This definition means that $g$ has the same image as $f$ but the arc between $t_0$ and $t_1$ is traveled along in the reverse direction. Since $g$, having the same image and length as $f$, is an optimal curve, which satisfies $g(t_0)=g(t_1)$, we have $g'_\ell(t_0)= g'_r(t_0)$ and $g'_\ell(t_1)= g'_r(t_1)$. On the other hand, by the definition of $g$, we know that $f'(t_0)=g'_\ell(t_0)=-g'_\ell(t_1)$ and $f'(t_1)=g'_r(t_1)=-g'_r(t_0)$. Hence, $f'(t_0)=-f'(t_1).$
\end{proof}

We introduce the set
\begin{equation*}
D= \Bigl\{ t\in[0,1) \mid \mbox{Card}\bigl( f^{-1}(\{f(t)\} )\cap [0,1) \bigr) \geq 2
\Bigr\}.
\end{equation*}

\begin{lem}\label{lem:Gronwall}
If  $f(t)$, $t\in(0,1)$, is a multiple point of $f:[0,1]\to \R^2$, then $t$ cannot be right- or left-isolated:\\
for all $t\in D\cap(0,1)$, for all $\e>0$, $(t,t+\e)\cap D \not=\emptyset$ and $(t-\e,t)\cap D\not=\emptyset$.
\end{lem}

\begin{proof}
Let $t_0\in D\cap(0,1)$. Assume that there exists $\e>0$ such that $(t_0,t_0+\e)\cap D =\emptyset$ or $(t_0-\e,t_0)\cap D=\emptyset$. We will show that this leads to a contradiction.
Without loss of generality, up to considering $t\mapsto f(1-t)$, we assume that $(t_0-\e,t_0)\cap D =\emptyset$. Let $t_1\in[0,1)$ such that $t_0\not=t_1$ and $f(t_0)=f(t_1)$.
By Lemma \ref{lem:tangent}, one has $f'_\ell(t_0)=- f'_r(t_1)$.

Let $$y=\frac{f'_r(t_1)}{L}$$ and define  the functions $\alpha$ and $\beta$ by
\begin{align*}
\alpha(t)&=\langle f(t)-f(t_1),y \rangle \mbox{ for } t\in [t_1,t_1+\e)
\\ \beta(t)&=\langle f(t)-f(t_0), y\rangle \mbox{ for } t\in (t_0-\e,t_0].
\end{align*}
Notice, since $f(t_0)=f(t_1)$, that $\alpha$ and $\beta$ are restrictions, to $[t_1,t_1+\e)$ and $(t_0-\e,t_0]$ respectively, of the same function. Nevertheless, this notation $\alpha$, $\beta$  were chosen for readability.

The functions $\alpha$ and $\beta$ satisfy the following properties:
\begin{itemize}
	\item $\alpha$ is right-differentiable and  $\alpha'_r(t)=\langle f'_r(t),y\rangle$ for every $t\in[t_1,t_1+\e)$. Since $\alpha '_r(t_1)=L>0$ and $\alpha'_r$ is right-continuous, there exists $\delta\in(0,\e)$, such that $\alpha '_r(t)\geq \delta L$ for every  $t\in[t_1,t_1+\delta]$.
	\item $\beta$ is left-differentiable and $\beta '_\ell(t)=\langle f_\ell'(t), y\rangle$	for every $t\in(t_0-\e,t_0].$ Since $\beta '_\ell(t_0)=-L<0 $ and $\beta_\ell '$ is left-continuous, there exists $\delta'\in(0,\e)$ such that $\beta_\ell '(t)\leq -\delta'L$ for every $t\in [t_0 - \delta', t_0].$
\end{itemize}

Without loss of generality, we may assume that $\delta'=\delta$, since it suffices to pick  the smallest of both values to have the properties on $\alpha'_r$ and $\beta'_\ell$.
In particular, we see that 
\begin{itemize}
	\item $\alpha$ is a bijection  from $[t_1,t_1+\delta]$ onto its image $\alpha([t_1,t_1+\delta])=[0,a],$ where $a:=\alpha(t_1+\delta)>0,$
	\item  $\beta$ is a bijection from $[t_0-\delta,t_0]$ onto its image $\beta([t_0-\delta,t_0])=[0,b]$, where $b:=\beta(t_0-\delta)>0$.
\end{itemize}  
We denote by $\alpha^{-1}$ and $\beta^{-1}$  their inverse functions.

Let $z\in\R^2$ be such that $\|z\|=1$ and $\langle z, y\rangle =0.$
For every $t\in(t_1,\alpha^{-1}(b)],$ we have $\langle f(t)-f(\beta^{-1}(\alpha(t))),y\rangle=0$. Then, we may write $ f(t)-f(\beta^{-1}(\alpha(t)))=\langle f(t)-f(\beta^{-1}(\alpha(t))),z\rangle z$. Moreover, for $t\in(t_1,\alpha^{-1}(b)]$, since there are no further multiple point before $t_0$, $f(t)-f(\beta^{-1}(\alpha(t)))\neq 0$.  Thus, there exists $\sigma\in\{-1,1\}$ such that $$\frac{f(t)-f(\beta^{-1}(\alpha(t)))}{\|f(t)-f(\beta^{-1}(\alpha(t)))\|}=\sigma z .$$ We suppose, without loss of generality, that the vector $z$ was chosen such that $\sigma=1.$
Now, let us show that, for $t\in(t_1,\alpha^{-1}(b)],$ $$\langle z,f'_r(t)\rangle\leq \frac 1{2\lambda}\sup_{t_1\leq s\leq t}\|f(s)-f(\beta^{-1}(\alpha(s)))\|.$$
Since $\langle z,f'_r(t_1)\rangle=0$, we have, according to Theorem \ref{theo:main}, 
\begin{align*}
\langle z,f'_r(t)\rangle&=\langle z,f'_r(t)-f'_r(t_1)\rangle\\&=\int_{(t_1,t]} \langle z,f''(ds)\rangle\\
&=-\frac{1}{\lambda} \E\left[\langle X-f(\hat{t}),z \rangle \I_{\{t_1<\hat{t}\leq t\}}\right]\\
&=-\frac{1}{\lambda} \E\left[\left\langle X-f(\hat{t}),\frac{f(\hat{t})-f(\beta^{-1}(\alpha(\hat{t})))}{\|f(\hat{t})-f(\beta^{-1}(\alpha(\hat{t})))\|} \right\rangle \I_{\{t_1<\hat{t}\leq t\}}\right]
\end{align*}
Besides, for $t\in[0,1]$, starting from $$\|X-f(t)\|^2= 
\|X-f(\hat{t})\|^2+\|f(\hat{t})-f(t)\|^2+2\langle X-f(\hat{t}),f(\hat{t})-f(t)\rangle,$$ we deduce, by optimality of $\hat{t}$, the inequality $$-\langle X- f(\hat{t}),f(\hat{t})-f(t) \rangle\leq \frac{1}{2}\|f(\hat{t})- f(t)\|^2\quad a.s.$$
Hence, we obtain
\begin{align}\label{eq:majr}
\langle z,f'_r(t)\rangle&\leq \frac{1}{2\lambda}\E\left[\|f(\hat{t})-f(\beta^{-1}(\alpha(\hat{t})))\|\I_{\{t_1<\hat{t}\leq t\}}\right]\nonumber\\
&\leq \frac{1}{2\lambda} \sup_{t_1<s\leq t}\|f(s)-f(\beta^{-1}(\alpha(s)))\|.
\end{align}
Similarly, we get, for every $t\in[\beta^{-1}(a),t_0)$,	\begin{equation}\label{eq:majl}
\langle z,f'_\ell(t)\rangle
\leq \frac{1}{2\lambda} \sup_{t\leq s< t_0}\|f(s)-f(\alpha^{-1}(\beta(s)))\|.
\end{equation} This may be seen for instance by considering the optimal curve parameterized in the reverse direction  $t\mapsto f(1-t)$.
For $x\in[0,a\land b)$, let $D(x)=f(\alpha^{-1}(x))-f(\beta^{-1}(x))$. This function $D$ is right-differentiable and 
$$ D'_r(x)=\frac{f'_r(\alpha^{-1}(x))}{\alpha'_r(\alpha^{-1}(x))}-\frac{f'_\ell(\beta^{-1}(x))}{\beta'_\ell(\beta^{-1}(x))}.$$
Moreover, $\alpha_r '(\alpha^{-1}(x))\geq \delta L$ and $-\beta '_\ell(\beta^{-1}(x))\geq \delta L$, so that 
\begin{align*}
\langle D'_r(x),z\rangle&\leq \frac{1}{\delta L}(\langle z, f'_r(\alpha^{-1}(x))\rangle+\langle z, f'_\ell(\beta^{-1}(x))\rangle)\\&\leq \frac{1}{\delta L\lambda}\sup_{u\leq x}\|D(u)\|.
\end{align*}For the last inequality, we used the upper bounds \eqref{eq:majr} and \eqref{eq:majl} together with the monotony of $\alpha$ and $\beta$. Observe, since $z=\frac{D(x)}{\|D(x)\|}$, that 
$\langle D'_r(x),z\rangle$ is the right-derivative of $\|D(x)\|$. As $D(0)=0$, the Gronwall Lemma implies that $D(x)=0$ for all $x\in[0,a\land b)$, which yields a contradiction, since  the considered  multiple point is supposed to be left-isolated.
\end{proof}

We may now state the injectivity result in dimension 2, for open and closed curves.

\begin{pro}
	\begin{enumerate}
		\item[$(i)$] If $\mathcal C_L=\{\phi\in[0,1]\to\R^2,\L(\phi)\leq L\}$, then $f$ is injective.
		\item[$(ii)$] If $\mathcal C_L=\{\phi\in[0,1]\to\R^2,\L(\phi)\leq L,\phi(0)=\phi(1)\}$, then either $f$ restricted to $[0,1)$ is injective or $\mbox{Im} f$ is a segment.
	\end{enumerate}
\end{pro}

\begin{proof}
$(i)$ $\mathcal C_L=\{\phi\in[0,1]\to\R^2,\L(\phi)\leq L\}$.\\
Thanks to Lemma \ref{lem:lastdouble}, if $f$ has multiple points, there exists a last multiple point. As such, this multiple point is right-isolated. However, by Lemma \ref{lem:Gronwall}, this cannot happen. So, $f$ is injective.

$(ii)$  $\mathcal C_L=\{\phi\in[0,1]\to\R^2,\L(\phi)\leq L,\phi(0)=\phi(1)\}$.\\ 
We assume that $f$ restricted to $[0,1)$ is not injective. So, our aim is to prove that $\mbox{Im} f$ is a segment.
As $f$ is supposed not to be injective, the set $D=\{ t\in[0,1) \mid \mbox{Card} ([0,1) \cap f^{-1}(\{f(t)\})) \ge 2 \}$ is non-empty. Without loss of generality, we can assume that $D\cap(0,1)\not=\emptyset$. Indeed, if $D=\{0\}$, we can replace $f$ by the curve $t\mapsto f((t+1/2) \text{ mod } 1)$ for which $D=\{1/2\}$.

Let us show that $D$ is dense in $(0,1)$. Proceeding by contradiction, we assume that there exists a non-empty open interval $(a,b)\subset(0,1)$ such that $D\cap (a,b)=\emptyset$.
Since $D\cap (0,1)\not=\emptyset$, one has $D\cap (0,a] \not=\emptyset$ or $D\cap [b,1) \not=\emptyset$.
Consider the case where $D\cap [b,1) \not=\emptyset$. Define $\beta=\inf( D\cap [b,1) )$. There exist two sequences $(t_k)_{k\ge 1}\subset D$ and $(s_k)_{k\ge 1}\subset D$ such that $t_k\downarrow \beta$, $f(t_k)=f(s_k)$ and $s_k\not= t_k$ for all $k\geq 1$.
Up to an extraction, $s_k$ converges to a limit $s\in[0,1]$. If $\beta\not= s$ then $\beta\in D$ is left-isolated which is impossible by Lemma \ref{lem:Gronwall}. Thus $s=\beta$ and consequently $s_k \ge \beta$ for $k$ large enough. This yields $f'_r(s_k) \to f'_r(\beta)$. Besides, for all $k$, $f'_r(t_k) \to f'_r(\beta)$ and, by Lemma \ref{lem:tangent}, $f'_r(t_k)=-f'_r(s_k)$, which contradicts the fact that $f$ has speed $L$.
The case where $D\cap (0,a] \not=\emptyset$ is similar.

The next step is to prove that the set $[0,1)\setminus D$ is finite.
Let $t\in(0,1)\setminus D$. Since $D$ is dense, there exists a sequence $(t_k)_{k\geq 1}\in D$ such that $t_k \downarrow t$. For every $k\geq 1$, there exists $s_k\neq t_k$ such that $f(t_k)=f(s_k)$. If $s\in[0,1]$ is a limit point of $(s_k)$, then $f(t)=f(s)$ which implies $t=s$ since $t\notin D$ and $t\not=0$. Therefore $\lim_{k\to\infty} s_k=t$.
Up to an extraction, we may assume that $(s_k)$ converges increasingly or decreasingly to $t$. 
By Lemma \ref{lem:tangent}, one has $f'(t_k)=-f'(s_k)$ for $k$ large enough.
If $s_k\downarrow t$, one obtains a contradiction: $f'_r(t)=\lim_k f'_r(t_k)=-\lim_k f'_r(s_k)=-f'_r(t)$. Thus $s_k\uparrow t$ and one gets $f'_r(t)=-f'_\ell(t)$. This means that $f(t)$ is a cusp. Since $\|f''\|([0,1])<\infty$, there are only a finite number of such points.

Observe that, as a consequence of Lemma \ref{lem:tangent}, 
for every $t\in[0,1)$, $\mbox{Card} ([0,1) \cap f^{-1}(\{f(t)\}))<3$.
Indeed, if a point has multiplicity at least 3, that is there exist $0\le t_1<t_2<t_3<1$ such that 
$f(t_1)=f(t_2)=f(t_3),$ then, on the one hand, $f'_r(t_1)=-f'(t_2)=-f'(t_3)$,  and on the other hand,  $f'(t_2)=-f'(t_3)$. Thus, one obtains again a contradiction: $f'_r(t_1)=f'(t_2)=f'(t_3)=0$.
In other words, $D=\{ t\in[0,1) \mid \mbox{Card} ([0,1) \cap f^{-1}(\{f(t)\})) = 2 \}.$

We introduce the function $\phi:[0,1)\to[0,1)$, defined as follows: for $t\in[0,1)\setminus D$, set
$ \phi(t)=t$ and for $t\in D$, set $\varphi(t)=t'$ where $t'\in f^{-1}(\{f(t)\})$ and $t'\notin t$. Note that $\phi$ is an involution.

Let us show that the function $\varphi$ is continuous on $(0,1)\setminus\{\phi(0)\}$. 
First, observe that $f$ is derivable on $D\cap (0,1)$ by Lemma \ref{lem:tangent}, and that $f'$ is continuous on $D\cap (0,1)$ since $f'_r$ is right-continuous and $f'_\ell$ is left-continuous.
Let $t\in(0,1)$ such that $t\not=\phi(0)$ and let $(t_k)_{k\geq 1}$ be a sequence converging to $t$.
Let $s\in[0,1]$ be a limit point of $(\phi(t_k))$. Since $f(t_k)=f(\phi(t_k))$,  for all $k\geq 1$, one has $f(s)=f(t)$.
Necessarily, $s\in(0,1)$ since $t\not=\phi(0)$.
If $t\notin D$, one has $s=t=\phi(t)$. If $t\in D$, then $s\in\{t,\phi(t)\}$. Since $D\cap(0,1)$ is open, $t_k\in D$ for $k$ large enough, hence
$f'(\varphi(t_k))=-f'(t_k)$ for $k$ large enough. Thus $f'(s)=-f'(t)$ and consequently $s=\phi(t)$.

Let us show that $\phi$ is derivable on $D\cap (0,1) \setminus\{\phi(0)\}$ and $\phi'(t)=-1$ for all $t\in D\cap (0,1) \setminus\{\phi(0)\}$. Let $t\in D\cap (0,1)$, $t\not=\phi(0)$. For all $h\in\R$ such that $|h|< t\land(1-t)$, we have
\begin{align*}
f(t+h)-f(t) &= f(\varphi(t+h))-f(\varphi(t))\\
&= \int_{\phi(t)}^{\phi(t+h)} f'(s) ds\\
&= \bigl(\phi(t+h)-\phi(t)\bigr) \int_0^1 f'\bigl(\phi(t)+u(\phi(t+h)-\phi(t))\bigr) du.
\end{align*}
Besides, since $f'$ is continuous at the point $\phi(t)\in D\cap(0,1)$ and $\phi$ is continuous at the point $t$, one has $\lim_{h\to 0} \int_0^1 f'\bigl(\phi(t)+u(\phi(t+h)-\phi(t))\bigr) du= f'(\phi(t))=-f'(t)$. One deduces that
$\lim_{h\to 0} \bigl(\phi(t+h)-\phi(t)\bigr)/h=-1$.

Let us prove that $\phi(\phi(0)/2+t)=\phi(0)/2 + 1-t  \mod 1$ for all $t\in[-\phi(0)/2, 1-\phi(0)/2)$.
From the two previous steps, one deduces that
if $\phi(0)=0$, $\phi(t)=1-t$ for all $t\in(0,1)$, as desired, while, if $\phi(0)\in(0,1)$,
there exist two constants $c_1$ and $c_2$ such that 
$$ \phi(t)=c_1-t \quad \forall t\in(0,\phi(0)), \quad \phi(t)=c_2-t \quad \forall t\in(\phi(0),1).$$
It remains to prove that $c_1=\phi(0)$ and $c_2=1+\phi(0)$.
As $\phi$ takes its values in $[0,1)$, one has $\phi(0)\le c_1\le 1$ and $1\le c_2\le 1+\phi(0)$.
Moreover, since $\phi$ is a bijection,
$c_2-t \ge c_1$ for $t\ge \phi(0)$ or $c_2-t\le c_1-\phi(0)$ for $t\ge \phi(0)$, that is $c_2-1\ge c_1$ or $c_2\le c_1$. In the first case, one gets $c_1=\phi(0)$ and $c_2=1+\phi(0)$. In the second case, one gets $c_1=c_2=1$,
which is not possible: necessarily, $\phi(0)=1/2$, since otherwise $\phi(1-\phi(0))=\phi(0)$ which yields $1-\phi(0)=0$, and we see that the restriction of $f$ to $[0,1/2]$ is a closed curve with the same image as $f$, hence $f$ is not optimal.

	Finally, define the curve $\tilde f$ by
	$$\tilde f(t)= f\bigl((\phi(0)/2 +t) \mod 1\bigr).$$
	This curve $\tilde f$ has the same image as $f$ and, from the last step,
	$\tilde f(t)=\tilde f(1-t)$ for all $t\in[0,1]$.
	Let us show that $\mbox{Im} f$ is a segment. Otherwise, the curve $g$ defined by
	$$ g(t)=\tilde f(t) \quad\text{if $t\in[0,1/2]$}, \quad
	g(t)=\tilde f(1/2)+ 2(t-1/2)\bigl(\tilde f(1)-\tilde f(1/2)\bigr) \quad\text{if $t\in[1/2,1]$}$$
	satisfies $\L(g)<\L(f)$ and $\Delta(g)\le \Delta(f)$, since $\mbox{Im} f=\tilde f([0,1/2])$, thus $f$ cannot be optimal.

\end{proof}

\section{Examples of principal curves}

\subsection{Uniform distribution on an enlargement of a curve}

The purpose of this section is to study the principal curve problem for the uniform distribution on an enlargement of some generative curve.
For $A\subset\R^d$ and $r\ge 0$, we denote by 
$$A\oplus r=\left\{x\in\R^d \mid d(x,A) \le r\right\}$$
the $r$-enlargement of $A$.
 Under some conditions on the generative curve $f:[0,1]\to\R^d$, for $r$ small enough, it turns out that the image of an optimal curve with length $\L(f)$ for the uniform distribution on an $r$-enlargement of $\mbox{Im} f$ is necessarily $\mbox{Im} f$. More specifically, the radius $r$ must not exceed the reach of $\mbox{Im} f$.

The reach of a set $A\subset\R^d$ is  the supremum of the radii $\rho$ such that every point at distance at most $\rho$ of $A$ has a unique projection on $A$. 
More formally, following  \cite{F59}, we define
for $A\subset\R^d$
$$ \mbox{reach}(A)=\sup\left\{\rho\ge 0 \mid \forall x\in\R^d \quad d(x,A)\le \rho \Rightarrow \exists! a\in A\quad d(x,a)=d(x,A)\right\} \in[0,+\infty].$$
				
The question of the optimality of the generative curve when considering the uniform distribution on an enlargement has been first addressed in dimension $d=2$ in \cite{MT05}. Observe that related ideas can be found in \cite{GPVW}.  Our proof in arbitrary dimension $d\geq 1$ relies on arguments in \cite{F59}, which moreover allow to show uniqueness.

\begin{theo} \label{th}
	Let $f:[0,1]\to\R^d$ be a  curve. Suppose that $f$ is injective, differentiable, $f'$ is Lipschitz, and there exists $c>0$ such that $\|f'(t)\|\ge c$ for all $t\in[0,1]$. Then, the reach of $\mbox{Im} f$ is positive. Let $r\in (0, \mbox{reach}(\mbox{Im} f]$ and let $X$ be a random vector uniformly distributed on $ \mbox{Im} f\oplus r.$
	Consider a function $V:[0,\infty)\to[0,\infty)$ continuous, increasing and such that $V(0)=0$. Then, for every curve $g:[0,1]\to \R^d$ such that $\L(g)\le \L(f)$ one has
	$$ \E\big[V\left(d(X, \mbox{Im} f )\right)\big]\le \E\big[V\left(d(X, \mbox{Im} g )\right)\big].$$
	with equality if and only if $\mbox{Im} g=\mbox{Im} f$.
\end{theo}

The proof of the theorem is based on two lemmas.
For $k\ge 1$, $\lambda_k$ denotes the Lebesgue measure on $\R^k$ and $\alpha_k$  the volume of the unit ball in $\R^k$.
From \cite[Lemma 42]{MT05}, we have the next result.
\begin{lem} \label{maj}
	Let $A$ be a compact connected subset of $\R^d$ with $\mathcal H^1(A)<\infty$. Then for all $r\ge 0$ one has
	$$ \lambda_d(A \oplus r) \le \mathcal H^1(A) \alpha_{d-1} r^{d-1} + \alpha_d r^d.$$
\end{lem}

\begin{lem}
	Let $f:[0,1]\to\R^d$ be a curve. Suppose that $f$ is injective,  $f$ is differentiable, $f'$ is Lipschitz, and there exists $c>0$ such that $\|f'(t)\|\ge c$ for all $t\in[0,1]$.
	Then, the reach of $A=\mbox{Im} f$ is positive and for all $r\le \mbox{reach}(A)$ one has
	\begin{equation} \label{vol}
	\lambda_d\left(A\oplus r\right)= \L(f)\alpha_{d-1} r^{d-1} + \alpha_d r^d
	\end{equation}
	Moreover, 
	one has
	\begin{equation} \label{fron}
	\left\{x\in A\oplus r \mid d\left(x,\partial (A\oplus r)\right) \ge r\right\} \subset A.
	\end{equation}
\end{lem}

\begin{proof}
The assumptions on $f$ imply that there exists $\e>0$, a set $B\subset\R^d$ and a function $\phi: (-\e,1+\e)\to B$
such that $\phi$ is bijective, $\phi=f$ on $[0,1]$, $\phi$ is differentiable, $\phi'$ is Lipschitz and $\phi^{-1}$ is Lipschitz. From  \cite[Theorem 4.19]{F59}, we deduce that $\mbox{reach}(A)>0$.
For $r\in(0,\mbox{reach}(A))$, equality \eqref{vol} follows from \cite[Theorem 5.6 \& Remark 6.14]{F59}.
For $r=\mbox{reach}(A)$, we can write
$$\L(f)\alpha_{d-1} r^{d-1} + \alpha_d r^d
= \lim_{n\to\infty} \lambda_d\bigpar{A\oplus(r-1/n)}=\lambda_d\bigpar{\{x\in\R^d \mid d(x,A)<r\}}$$
and $\lambda_d\bigpar{A\oplus r}
\le \L(f)\alpha_{d-1} r^{d-1} + \alpha_d r^d$ by Lemma \ref{maj}. Thus, equality \eqref{vol} holds.

Now, we prove \eqref{fron}. Let $x\in A\oplus r$ such that $d\bigpar{x,\partial (A\oplus r)}\ge r$.
According to \cite[Corollary 4.9]{F59}, if $0<s<\mbox{reach}(A)$ and $A'_s=\{y\in\R^d\mid d(y,A)\ge s\}$ then
$$ d(y,A'_s)=s-d(y,A) \quad \text{whenever $0<d(y,A)\le s$}.$$
Suppose that $d(x,A)>0$, then for all $s\in [d(x,A), r)$ one has
$d(x,A)=s-d(x,A'_s)$. Since $\lim_{s\to r} d(x,A'_s)=d(x,A'_r)=d\bigpar{x,\partial (A\oplus r)}$, one gets $d(x,A)\le 0$. This proves that $x\in A$.

\end{proof}

\begin{proof}[Proof of Theorem \ref{th}]
We set $A=\mbox{Im} f$ and $B=\mbox{Im} g$. On the one hand, denoting by $V^{-1}$ the inverse of $V:[0,\infty)\to[0,V(\infty))$,
\begin{align*}
\E\bigcro{V(d(X,B))} &= \int_0^\infty  \P\bigcro{ V(d(X,B)) >t} dt\\
&= \int_{0}^{V(\infty)}  \Bigpar{1-\P\bigpar{ d(X,B) \le V^{-1}(t)}} dt\\
&=\int_0^{V(\infty)} \Bigpar{1-\frac{ \lambda_d\bigpar{(B\oplus V^{-1}(t))\cap (A\oplus r)} }{ \lambda_d(A\oplus r)}} dt.
\end{align*}
On the other hand, for $t\le V(r)$, by Lemma \ref{maj} and equation \eqref{vol}, one gets
\begin{multline} \label{ine}
\lambda_d\bigpar{(B\oplus V^{-1}(t))\cap (A\oplus r)} \le
\lambda_d\bigpar{B\oplus V^{-1}(t)} \le \mathcal H^1(\mbox{Im} g) \alpha_{d-1} V^{-1}(t)^{d-1} + \alpha_d V^{-1}(t)^d\\
 \le \L(f) \alpha_{d-1} V^{-1}(t)^{d-1} + \alpha_d V^{-1}(t)^d=\lambda_d\bigpar{A\oplus V^{-1}(t)}.\end{multline}
Therefore, for all $t\in[0,V(\infty))$,
$$1- \frac{ \lambda_d\bigpar{(B\oplus V^{-1}(t))\cap (A\oplus r)} }{ \lambda_d(A\oplus r)}
\ge \biggcro{1- \frac{\lambda_d\bigpar{A\oplus V^{-1}(t)}}{ \lambda_d\bigpar{A\oplus r}}}_+ .$$
Consequently,
\begin{align*}
\E\bigcro{V(d(X,B))}
&\ge \int_0^{V(r)}  \biggpar{1- \frac{\lambda_d\bigpar{A\oplus V^{-1}(t)}}{ \lambda_d\bigpar{A\oplus r}}} dt
=\E\bigcro{V(d(X,A))}.
\end{align*}
Suppose that $\E\bigcro{V(d(X,B))}=\E\bigcro{V(d(X,A))}$. Then, we have
$$ \lambda_d\bigpar{A\oplus r} =\lambda_d\bigpar{(B\oplus V^{-1}(t))\cap (A\oplus r)}\; \mbox{$dt-$a.e. on } [V(r),+\infty).$$
 By right continuity with respect to $t$, we obtain that
$\lambda_d\bigpar{A\oplus r} =\lambda_d\bigpar{(B\oplus r)\cap (A\oplus r)}.$
From the chain of inequalities \eqref{ine} with $t=V(r)$, we deduce that $\L(f)=\mathcal H^1(\mbox{Im} g)$ and
$$ \lambda_d\bigpar{(A\oplus r)\cap (B\oplus r)^c } =\lambda_d\bigpar{(A\oplus r)^c\cap (B\oplus r) }=0.$$
Let us show that $A\oplus r=B\oplus r$. Suppose that $(A\oplus r)\cap (B\oplus r)^c\not=\emptyset$. Then one can find $x\in\R^d$ and $a\in A$ such that $d(x,a)\le r$ and $d(x,B)>r$. Set $y=x-\e(x-a)$ where $0<\e\le 1$ and
$\e < d(x,B)/r-1$. One has
$d(y,a)\le r-\e r<r$ and $d(y,B)>d(x,B)-\e r>r$. Thus $y$ belongs to the interior of $(A\oplus r)\cap (B\oplus r)^c$ which implies $\lambda_d\bigpar{(A\oplus r)\cap (B\oplus r)^c }>0$.
Therefore $(A\oplus r)\cap (B\oplus r)^c=\emptyset$. Similarly one can prove that $(A\oplus r)^c\cap (B\oplus r)=\emptyset$.

Finally, from \eqref{fron}, we deduce that
$$B\subset \{x\in B\oplus r \mid d\bigpar{x,\partial (B\oplus r)} \ge r\}
=\{x\in A\oplus r \mid d\bigpar{x,\partial (A\oplus r)} \ge r\}\subset A.$$
Since $\L(f)=\mathcal H^1(\mbox{Im} g)$, this implies that $A=B$.
\end{proof}

\subsection{Uniform distribution on a circle}\label{section:exe}

In this section, we investigate the  principal curve problem for a particular distribution, the  uniform distribution on a circle.

\begin{pro}Consider the unit circle centered at the origin with parameterization given by  $$g(t)=(\cos(2\pi t), \sin(2\pi t))$$ for $t\in[0,1]$. Let $U$ be a uniform random variable on $[0,1]$ and let $X=g(U)$.
Then, for every $L<2\pi$,  the circle centered at the origin with radius $\dfrac{L}{2\pi}$ is the unique closed principal curve with length $L$ for $X$.
\end{pro}

\begin{proof}
	Let $f:[0,1]\to\R^2$ be  an optimal closed curve with length $L$.
	We denote by $K$ the convex hull of $\mbox{Im} f$. Since $\mbox{Im} f$ is compact,  $K$ is a compact convex set (consequence of Caratheodory's theorem; see, e.g., \cite{hiriart2012}).
	Notice that $\mbox{Im} f$ is included in the unit disk: indeed, if not, since $f$ is a closed curve, with $\L(f)<2\pi$, there exist $u_1$ and $u_2$, such that $f(u_1)$ and $f(u_2)$ belong to the unit circle and the arc $t\in(u_1,u_2)\mapsto f(t)$ is outside the disk, which is not optimal since replacing this arc by the corresponding unit circle arc yields a better and shorter curve. In turn, the convex hull $K$ is also included in the unit disk, by convexity of the latter.
			Let $\pi_K: \R^2 \to K$ denote the projection onto $K$ et define the curve $h$ by $h(t)=\pi_K(g(t))$ for $t\in[0,1]$.
	By this definition of $h$ as projection of the unit circle on a set included in the unit disk containing $\mbox{Im} f$, we have \begin{equation*}
	\Delta(h)\leq\Delta(f).
	\end{equation*}
	\begin{itemize}
		\item Let us prove that $h$ has length at most $L$.
	First, note that $h$ has finite length, since $\pi_K$ is Lipschitz. By properties of the projection on a closed convex set, we know that the set of points of $\R^2$ projecting onto a given element of the boundary $\partial K$ of $K$ is a cone. This ensures that $h:[0,1]\to\partial K$ is onto, because a cone with vertex in the unit disk intersects the unit circle $\mbox{Im} g$ at least once. More specifically, if the cone reduces to a half-line (degenerated case), then it intersects  $\mbox{Im} g$ exactly once. Otherwise, the cone is the region delimited  by two distinct half-lines with common origin in the disk, and, thus, contains an infinity of such distinct half-lines, each of them intersecting $\mbox{Im} g$ once. Hence, for every  $v\in \mbox{Im} h$, there is either one $t$ such that $v=h(t)$, or an infinity.

	We will use Cauchy-Crofton's formula on the length of a  curve (for a proof, see, e.g., \cite{ayari}). Let $d_{r,\theta}$ denote the line with equation $x\cos\theta+y\sin\theta=r$. For every  curve $\phi=(\phi^1,\phi^2)$, if $$N_\phi(r,\theta)=\mbox{Card}(\{t\in[0,1],\phi(t)\in d_{r,\theta}\})=\mbox{Card}(\{t\in[0,1],\phi^1(t)\cos\theta+ \phi^2(t)\sin\theta=r\}),$$  then the length of $\phi$ is given by $$\frac 1 4 \int_0^{2\pi}\int_{-\infty}^{\infty}N_\phi(r,\theta)drd\theta .$$
	
	Let us compare $N_h(r,\theta)$ and $N_f(r,\theta)$ for $(r,\theta)\in \R\times [0,2\pi]$. To begin with, note that $N_h(r,\theta)$ is finite almost everywhere since $h$ has finite length. So, we  need only consider the cases where $N_h(r,\theta)$ is finite. This allows to exclude the points $v\in \mbox{Im} h$ such that $h^{-1}(\{v\})$ is infinite, as well as the cases where a line $d_{r,\theta}$ and $\mbox{Im} h$ have a whole segment in common.
	Observing that, if the line $d_{r,\theta}$ does not intersect $\mbox{Im} h$, then it does not intersect $\mbox{Im} f$ either, since $\mbox{Im} h$ is the boundary of the convex hull of $\mbox{Im} f$, it remains to look at the two following cases for comparing $N_h(r,\theta)$ and $N_f(r,\theta)$.
\begin{itemize}
	\item If the line $d_{r,\theta}$ intersects $\mbox{Im} h$ at a single point, then this point belongs to $\mbox{Im} f$.
	\item If the line $d_{r,\theta}$ intersects $\mbox{Im} h$ at exactly two points, then $\mbox{Im} f$ crosses the line. If $\mbox{Im} f$ were located on one side of the line, $K$ were not the convex hull. Since $f$ is a closed curve, $\mbox{Im} f$ crosses the line at least twice.
\end{itemize}
So, $N_h(r,\theta)\leq N_f(r,\theta)$ almost everywhere, that is $\L(h)\leq \L(f)=L$.

	\item Now, observe that 	$ \mbox{Im} h\subset \mbox{Im} f$. Indeed, otherwise, there exists $t\in[0,1]$ such that $h(t)\notin \mbox{Im} f$, which means that $d(g(t),\mbox{Im} f)>d(g(t),K)$. By continuity this implies that $d(g(s),\mbox{Im} f)>d(g(s),K)$ for all $s$ in a non-empty open set and one obtains that $\Delta(h)<\Delta(f)$.
	By optimality of $f$, this is not possible since  $\L(h)\le L$.

	\item Since $ \mbox{Im} h\subset \mbox{Im} f$ and $\L(f)=L$, to obtain that $\mbox{Im} f$ is the circle with  center $(0,0)$ and radius $L/2\pi$, it remains to show that $\mbox{Im} h$ is the circle with  center $(0,0)$ and radius $L/2\pi$.  Let $\theta\in[0,1]$ and let $A_\theta:\R^2\to\R^2$ denote the  rotation with center $(0,0)$ and angle $2\pi\theta$.
We set $ h_\theta(t)=\pi_{A_\theta(K)}(g(t))$, for every $t\in[0,1]$. Since $ h_\theta(t)=A_\theta\circ \pi_K
		( A_\theta^{-1}(g(t)))= A_\theta\circ\pi_K( g(t-\theta))$, $ h_\theta$ is a curve with same length as $h$.
Moreover, $A_\theta(X)$ has the same distribution as $X$, so that
	\begin{align*}
	\E\left[\|X-\pi_{A_\theta(K)}(X)\|^2\right]&=\E\left[\|A_\theta(X)-\pi_{A_\theta(K)}(A_\theta(X))\|^2\right]\\
	&=\E\left[\|A_\theta(X)-A_\theta(\pi_{K}(X))\|^2\right]\\
	&=\E\left[\|X-\pi_{K}(X)\|^2\right].
	\end{align*}
	By strict convexity, we deduce from this equality that, 
	if  $\P\left(\pi_{A_\theta(K)}(X)\not=\pi_K(X)\right)>0$, then
	$$ \E\left[\|X-(\pi_K(X)+\pi_{A_\theta(K)}(X))/2\|^2\right] < \Delta(h).$$
	Since the random variable  $(\pi_K(X)+\pi_{A_\theta(K)}(X))/2$ takes its values in the image of the  curve $(h+ h_\theta)/2$ with length  smaller than $\L(h)\le L$, that is not possible. Consequently, $\pi_{A_\theta(K)}(X)=\pi_K(X)$ almost surely. In other words,  
	$\pi_{A_\theta(K)}(g(t))=h(t)$ for almost every $t\in[0,1]$, and, thus, by continuity, $h_\theta(t)=h(t)$ for every $t\in[0,1]$.
	For $t\in[0,1]$, let $\theta=t$. We have $h(t)=h_t(t)=A_t\circ \pi_K(g(0))=A_t( h(0) )$. 
	Since $\Delta(f)=\Delta(h)$ and $\L(h)\le L$, $\L(h)=L$.
	Hence, $\mbox{Im} h$ is the circle with center $(0,0)$ and  radius $L/2\pi$.

	\end{itemize}
\end{proof}

\begin{rem}
	Observe that  radial symmetry of a distribution is not sufficient to guarantee that a given circle will be a constrained principal curve for this distribution. Let us exhibit two counterexamples.
	\begin{itemize}
		\item Let $p>0$ and let $\mathcal U$ denote the uniform distribution on the unit circle. Consider a random variable $X$ taking its values in $\R^2$, distributed according to the mixture distribution $$p\delta_{(0,0)}+(1-p)\mathcal U,$$ where $\delta_{(0,0)}$ stands for the Dirac mass at the origin $(0,0)$. Then, for every circle with center $(0,0)$ and radius $r\in(0,1]$, because of the atom at the origin, the projection of $X$ on the circle is not unique almost surely, which implies, thanks to Proposition \ref{prop:hatXneg}, that none of these circles may be a constrained principal curve for $X$. 
	
		\item We consider the case where $X$ is a standard Gaussian random vector in $\R^2$. Lemma \ref{lem:noneq} ensures that  the circle with center $(0,0)$ and radius $\E[\|X\|]=\sqrt{\pi/2}$ cannot be a constrained principal curve for $X$ because it   is self-consistent.
	\end{itemize}
	
\end{rem}

\bibliography{courbiblio}

\bibliographystyle{plainnat}

\Addresses
\end{document}